\documentclass[a4paper,notitlepage,twoside,reqno,11pt]{amsart}

\usepackage{anysize}
\marginsize{3.4cm}{3.4cm}{3cm}{3cm}

\usepackage{bbm,pifont,latexsym}

\usepackage{dcolumn,indentfirst}
\usepackage[hypertexnames=false,bookmarksopen=true,linktocpage=true,pdfstartview={XYZ null null 1.25}]{hyperref}
\usepackage{amsmath,amssymb,amscd,amsthm,amsfonts,mathrsfs}
\usepackage{color,graphicx,xcolor,graphics,subfigure,extarrows,caption2}
\usepackage{pdfpages,titletoc}

\usepackage{autonum}

\newtheorem{thm}{Theorem}[section]
\newtheorem{lem}[thm]{Lemma}
\newtheorem{cor}[thm]{Corollary}

\theoremstyle{definition}
\newtheorem*{defi}{Definition}

\newtheorem*{rmk}{Remark}

\newcommand{\EC}{\widehat{\mathbb{C}}}
\newcommand{\C}{\mathbb{C}}

\newcommand{\N}{\mathbb{N}}
\newcommand{\Q}{\mathbb{Q}}
\newcommand{\R}{\mathbb{R}}
\newcommand{\T}{\mathbb{T}}
\newcommand{\Z}{\mathbb{Z}}

\newcommand{\ii}{\textup{i}}
\newcommand{\Crit}{\textup{Crit}}
\newcommand{\diam}{\textup{diam}}

\newcommand{\HT}{\textup{HT}}
\newcommand{\Bif}{\textup{Bif}}

\newcommand{\MB}{\mathcal{B}}
\newcommand{\MH}{\mathcal{H}}
\newcommand{\MP}{\mathcal{P}}
\newcommand{\MU}{\mathcal{U}}

\makeatletter\@addtoreset{equation}{section}\makeatother

\begin{document}

\author[Weiwei Su]{Weiwei Su}
\address{School of Mathematics, Nanjing University, Nanjing 210093, P. R. China}
\email{suweiwei2024@163.com}

\author[Fei Yang]{Fei Yang}
\address{School of Mathematics, Nanjing University, Nanjing 210093, P. R. China}
\email{yangfei@nju.edu.cn}

\title[Carpet Julia sets of transcendental entire functions]{Sierpi\'{n}ski carpet Julia sets of transcendental entire functions}

\begin{abstract}
We prove that Sierpi\'{n}ski carpet parameters are dense in the bifurcation loci of some natural families of finite type transcendental entire functions having exactly one active critical orbit.
Moreover, some specific transcendental entire functions with Sierpi\'{n}ski carpet Julia sets are constructed such that they contain either an attracting basin, a parabolic basin, a Siegel disk or wandering domains.
The proofs are mainly based on applying the polynomial-like renormalization and some results on the local connectivity of transcendental Julia sets obtained before.
\end{abstract}

\subjclass[2020]{Primary: 37F10; Secondary: 37F44, 37F46}

\keywords{Julia set; Sierpi\'{n}ski carpet; local connectivity; transcendental entire function; bifurcation locus}

\date{\today}



\maketitle

\section{Introduction}\label{introduction}

\subsection{Backgrounds}

According to \cite{Why58}, a subset $S$ in the Riemann sphere $\EC$ is called a \emph{Sierpi\'{n}ski carpet} (\emph{carpet} in short) if $S$ is compact, connected, locally connected, has empty interior and the complement of $S$ consists of infinitely many components which are bounded by pairwise disjoint Jordan curves. All Sierpi\'{n}ski carpets are homeomorphic to each other.

It was known that as Julia sets of rational maps, Sierpi\'{n}ski carpets play an important role in complex dynamics (see \cite{McM94a}, \cite{BLM16}, \cite{DL25}).
The first example of Sierpi\'{n}ski carpet Julia sets was found by Milnor and Tan Lei \cite[Appendix F]{Mil93}. They constructed such Julia sets in quadratic rational maps having two super-attracting cycles with periods $3$ and $4$. In \cite{DFGJ14}, Devaney, Fagella, Garijo and Jarque proved the existence of Sierpi\'{n}ski carpet Julia sets in quadratic rational families having super-attracting cycles with various periods.
As a singular perturbation of the monomial $z\mapsto z^n$, the McMullen maps $z\mapsto z^n+\lambda/z^m$ ($n\geq 2$, $m\geq 1)$ was studied by Devaney, Look and Uminsky \cite{DLU05} and they proved that this family can produce a lot of hyperbolic components called \textit{Sierpi\'{n}ski holes} which correspond to carpet Julia sets (see also \cite{BDL05}). One may also refer to \cite[\S 5.6]{Pil94}, \cite{Mor00}, \cite{Ste08}, \cite{Loo10}, \cite{Yan18} and \cite{FY20} for some other rational maps having Sierpi\'{n}ski carpet Julia sets.

The Julia sets of polynomials have no buried points. Hence any polynomial Julia set cannot be a Sierpi\'{n}ski carpet.
However, the Julia sets of transcendental entire functions can be. Note that the Julia set $J(f)$ of any transcendental entire function $f$ is a closed subset of $\C$, which is not compact. Here we say that $J(f)$ is a Sierpi\'{n}ski carpet if $J(f)\cup\{\infty\}$ is (regarded as a subset of $\EC$).
The first carpet Julia sets of transcendental entire functions was constructed by Morosawa \cite{Mor99}. He proved that if $a>1$, then the Julia set of
\begin{equation}
f(z)=ae^a(z-(1-a))e^z
\end{equation}
is a Sierpi\'{n}ski carpet. See \cite{GJM11} and \cite[\S 2]{BFR15} for further study of this family and beyond.
Soon after, Bergweiler and Morosawa also found another transcendental entire family with carpet Julia sets (see \cite[Example 2]{BM02}).
Both of the carpet Julia sets constructed in \cite{Mor99} and \cite{BM02} contain escaping singular values.

In this paper, we show that as in the rational case, there are many transcendental entire families containing \textit{Sierpi\'{n}ski holes}, which are connected components of hyperbolic transcendental entire functions whose Julia sets are Sierpi\'{n}ski carpets. Moreover, we also construct some specific transcendental entire functions with Sierpi\'{n}ski carpet Julia sets such that they contain either a parabolic basin, a Siegel disk or wandering domains.

\subsection{Main results}

Let $f$ be a transcendental entire function. We use $F(f)$ and $J(f)$ to denote the \textit{Fatou set} and \textit{Julia set} of $f$ respectively. Let $CV(f)$ and $AV(f)$ be the \textit{critical values} and \textit{asymptotic values} of $f$ in $\C$ respectively. The singular set $S(f)$ of $f$ is the collection of all \textit{singular values} of the inverse function $f^{-1}$, which is the closure of  $CV(f)\cup AV(f)$ in $\C$. We call
\begin{equation}
\MB:=\{f:\C\to\C \text{ is transcendental entire}: S(f) \text{ is bounded}\}
\end{equation}
the \textit{Eremenko-Lyubich class} (compare \cite{EL92}).
The \textit{postsingular set} of $f$ is
\begin{equation}\label{equ:P-f}
\MP(f):=\overline{\bigcup_{n\geq 0}f^{\circ n}\big(S(f)\big)}.
\end{equation}
The entire function $f$ is called \textit{hyperbolic} if $f\in\MB$ and every element of $S(f)$ belongs to an attracting basin of $f$, or equivalently, $f$ is hyperbolic if and only if $\MP(f)$ is a compact subset of the Fatou set $F(f)$ (see \cite[Proposition 2.1]{BFR15}).

\begin{thm}\label{thm:carpet-hyper}
Let $f$ be a hyperbolic transcendental entire function without asymptotic values. Suppose that
\begin{enumerate}
\item every component of $F(f)$ contains at most one critical value and the multiplicity of the critical points of $f$ is uniformly bounded; and
\item the components of immediate attracting basins of $f$ have pairwise disjoint boundaries.
\end{enumerate}
Then $J(f)$ is a Sierpi\'{n}ski carpet.
\end{thm}

In fact, condition (a) in Theorem \ref{thm:carpet-hyper} implies that $J(f)$ is locally connected (see \cite[Corollary 1.9(a)]{BFR15}) and condition (b) further leads to $J(f)$ being a Sierpi\'{n}ski carpet. Hence Theorem \ref{thm:carpet-hyper} still holds if condition (a) is replaced by some other conditions which guarantee the local connectivity of $J(f)$ (see \cite[Corollaries 1.8 and 1.9(b)]{BFR15} for example). Moreover, if $f$ has exactly one immediate attracting basin and the period is one, then condition (b) is automatically satisfied.

\medskip
In the rational case, the hyperbolic components corresponding to Sierpi\'{n}ski carpet Julia sets are called Sierpi\'{n}ski holes.
Such concept can be adopted to the transcendental case naturally (see \S\ref{sec:carpet-hole} for the definition and Figure \ref{Fig:M-sin-2}).

\begin{figure}[!htpb]
  \setlength{\unitlength}{1mm}
  \centering
  \includegraphics[width=0.96\textwidth]{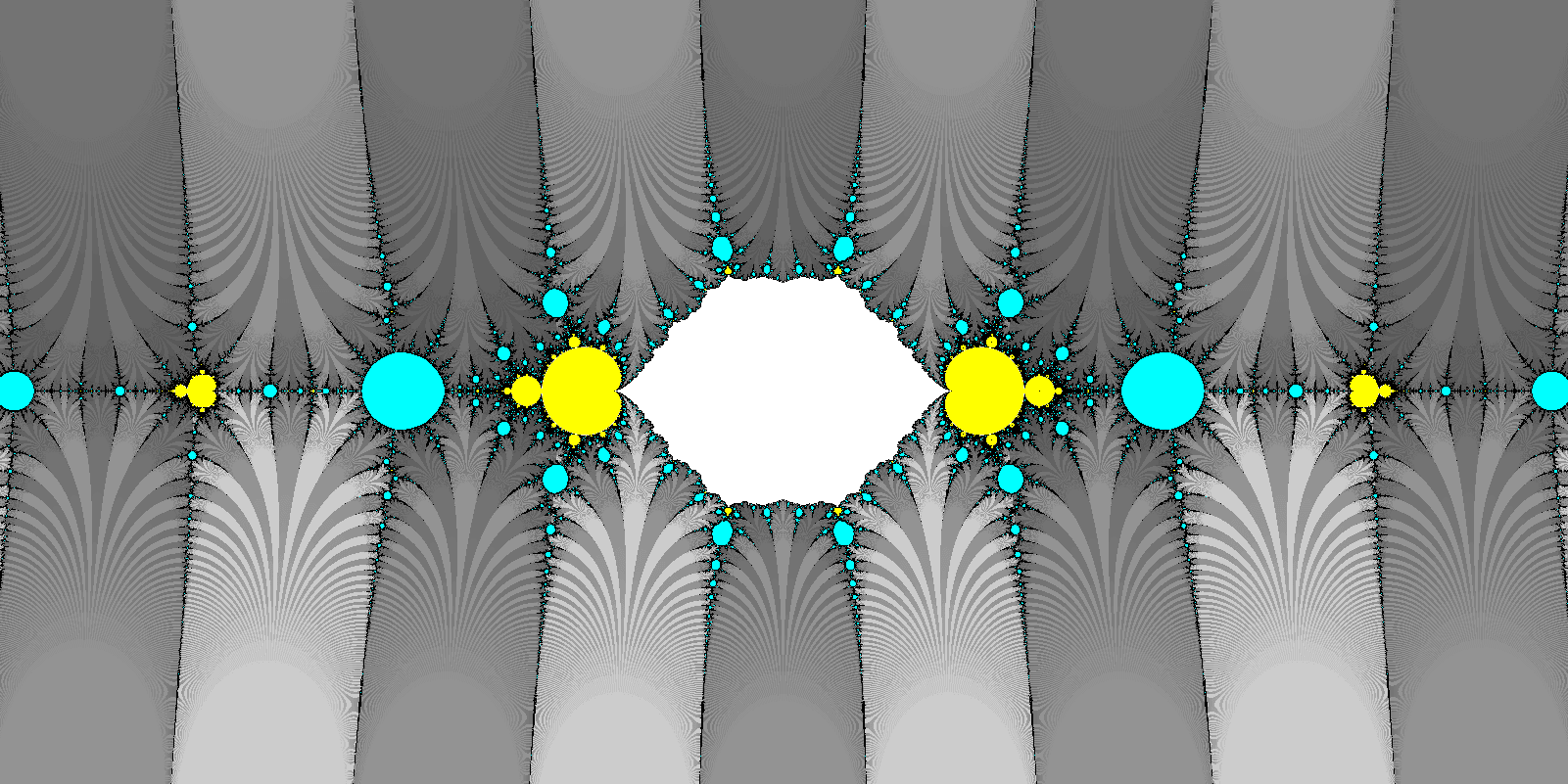}
  \caption{The $\lambda$-parameter plane of the family $z\mapsto\lambda\sin^2 z$, where $\lambda\in\C\setminus\{0\}$. Many Sierpi\'{n}ski holes are visible (the cyan and some yellow domains) and the carpet parameters are dense in the bifurcation locus. Figure range: $[-6.4,6.4]\times[-3.2,3.2]$.}
  \label{Fig:M-sin-2}
\end{figure}

\begin{cor}\label{cor:carpet-hole}
For $f_\lambda(z)=\lambda\sin^d z$, where $d\geq 2$ and $\lambda\in\C\setminus\{0\}$, if $\lambda$ lies in the attracting basin of $0$ but not in the immediate attracting basin, then $J(f_\lambda)$ is a Sierpi\'{n}ski carpet. In particular, for each $\lambda_k=k\pi$ with $k\in\Z\setminus\{0\}$, there is a Sierpi\'{n}ski hole in the parameter plane of the family $\{f_\lambda\}$ containing $\lambda_k$.
\end{cor}

A transcendental entire function $f$ is called \textit{finite type}, if it has finitely many \textit{singular values} (i.e., critical values or asymptotic values).
A holomorphic family $(f_\lambda)_{\lambda\in\Lambda}$ of finite type entire functions parameterized by a connected complex manifold $\Lambda$ is called \textit{$J$-stable} at $\lambda_0$, if there exists a neighborhood $U$ of $\lambda_0$ in $\Lambda$ such that the Julia set $J(f_\lambda)$ moves holomorphically in $U$ (see \cite{EL92} or \cite{ABF23} for more equivalent characterizations of $J$-stability of finite type entire functions).
The \textit{bifurcation locus} $\Bif(f_\lambda)$ of this family is the set of all $\lambda_0\in\Lambda$ such that $(f_\lambda)_{\lambda\in\Lambda}$ is not $J$-stable at $\lambda_0$.

\begin{thm}\label{thm:carpet-dense}
Let $(f_\lambda)_{\lambda\in\Lambda}$ be a natural family of finite type transcendental entire functions without asymptotic values, where $\Lambda \subset \mathbb{C}$ is a domain. Suppose that
\begin{enumerate}
\item the multiplicity of the critical points of $f_\lambda$ is uniformly bounded;
\item $c_{\lambda}$ is a marked active critical point which is not persistently a Picard exceptional value and $f_\lambda(c_\lambda)$ is the only active critical value; and
\item either the remaining critical values are all preperiodic in the Julia set, or the number of the remaining critical values is one and it lies in a fixed immediate attracting or parabolic basin.
\end{enumerate}
Then the Sierpi\'{n}ski carpet parameters are dense in $\Bif(f_\lambda)$.
\end{thm}

For the definition of natural family, see \cite{ABF23} or \S\ref{sec:natural-family}.
For any finite type entire function $f$, according to Eremenko and Lyubich \cite{EL92}, $f$ can be embedded in a finite dimensional complex manifold $M$ of dimension $\sharp S(f)+2$, where $S(f)\subset\C$  is the set of singular values of $f$.
Natural family can be viewed as subfamilies in this parameter space $M$.
According to \cite[Theorem 2.6]{ABF23}, a holomorphic family $(f_\lambda)_{\lambda\in\Lambda}$ of finite type entire functions is locally a natural family if $S(f_\lambda)$ and $f_\lambda^{-1}(S(f_\lambda))$ move holomorphically.
Hence it is easy to check that many simple families of finite type entire functions are natural, such as $f_\lambda(z)= \lambda \sin^d z$ etc, where $d\geq 1$.

\medskip
To show that a transcendental Julia set $J(f)$ is a Sierpi\'{n}ski carpet, it suffices to prove that $J(f)$ is locally connected and all Fatou components of $f$ are bounded by pairwise disjoint Jordan curves (see Lemma \ref{lem:carpet-criterion}).
The local connectivity of Julia sets of transcendental entire functions has been studied extensively recently, by constructing suitable expanding metrics near Julia sets. See \cite{Osb13}, \cite{BFR15}, \cite{ARS22}, \cite{Par22}, \cite{YZZ25}, \cite{QW25} and the references therein (see also \cite{BFJK25} for the study of local connectivity of Julia sets of transcendental meromorphic functions).
To prove Theorem \ref{thm:carpet-dense}, one of the main ingredients is to show that the Fatou components are bounded by pairwise disjoint Jordan curves.
For this, we will use polynomial-like renormalization theory to study the distribution of the (generalized) Mandelbrot set in the parameter space (based on the work of McMullen \cite{McM00b} and Astorg-Benini-Fagella \cite{ABF26}) and ``arrange" some unicritical polynomial dynamics in the Julia sets of transcendental entire functions such that they are Sierpi\'{n}ski carpets.

The transcendental entire functions with carpet Julia sets constructed in Theorem \ref{thm:carpet-dense} are all parabolic. In fact, we can also prove the following result.

\begin{thm}\label{thm:carpet-Siegel}
There are transcendental entire functions having Siegel disks whose Julia sets are Sierpi\'{n}ski carpets.
\end{thm}

It was known that all finite type entire functions have neither Baker domains nor wandering domains (see \cite{GK86}, \cite{EL92}).
If a transcendental entire function $f$ has an unbounded Fatou component, then $J(f)$ cannot be locally connected (see \cite{BD00b}, \cite{Osb13}).
Hence if $f$ has a Baker domain, then $J(f)$ cannot be a Sierpi\'{n}ski carpet.

\begin{thm}\label{thm:carpet-wandering}
There are transcendental entire functions having wandering domains whose Julia sets are Sierpi\'{n}ski carpets.
\end{thm}

Baker gave the first transcendental entire functions having wandering domains in 1970s and Herman gave the first simply connected wandering domains (see \cite[p.\,564]{Bak84}). To prove Theorem \ref{thm:carpet-wandering}, the wandering domains are necessarily simply connected and bounded. It was known that there are wandering domains in the family
\begin{equation}
f_\lambda(z)=z+\lambda\sin z+2n\pi,
\end{equation}
where $n\in\Z$ and $\lambda\in\C\setminus\{0\}$, based on the logarithmic lift of the self-dynamics of $\C\setminus\{0\}$ to that of $\C$ (see \cite[p.\,569]{Bak84}, \cite{FH09}). We shall prove that $J(f_\lambda)$ are Sierpi\'{n}ski carpets for certain $\lambda$ and any $n\in\Z$ (see Theorem \ref{thm:carpet-lift}).
For the classification of simply connected wandering domains, see \cite{BEFRS22}.

\medskip
This paper is organized as following:
In \S\ref{sec:natural-family}, we recall some basic definitions related to holomorphic family of entire functions, including $J$-stability, bifurcation locus, natural family and state the universality of the Mandelbrot sets in the parameter spaces.
In \S\ref{sec:carpet-hole}, we prove Theorem \ref{thm:carpet-hyper} based on Bergweiler-Fagella-Rempe's result and apply Theorem \ref{thm:carpet-hyper} to prove Corollary \ref{cor:carpet-hole}.
In \S\ref{sec:density}, we combine the universality of the Mandelbrot set in some transcendental families and results on the local connectivity to prove the density of carpet parameters of the families in Theorem \ref{thm:carpet-dense}.
In \S\ref{sec:carpet-Siegel}, we prove Theorem \ref{thm:carpet-Siegel} by constructing transcendental carpet Julia sets with Siegel disks based on some results of hairy Siegel disks of quadratic polynomials.
In \S\ref{sec:carpet-wandering}, we construct transcendental carpet Julia sets with wandering domains by applying the logarithmic lift construction and prove Theorem \ref{thm:carpet-wandering}.

\medskip
\noindent\textbf{Acknowledgements.}
The second author would like to thank N\'{u}ria Fagella and Anna Jov\'{e} for helpful explanations on Baker domains during a conference in Isaac Newton Institute for Mathematical Sciences in July 2026.
This work was supported by NSFC (Grant No.\,12571093).

\section{Mandelbrot copies in transcendental bifurcations}\label{sec:natural-family}

In this section, we give the definitions of ${J}$-stability, bifurcation locus, naturally family and state the universality of the Mandelbrot sets in the parameter spaces in some finite type entire families. Most contents can be found in \cite{DH85b}, \cite{McM94b}, \cite{McM00b}, \cite{ABF23} and \cite{ABF26}.

\subsection{${J}$-stability}
We first give the definitions of holomorphic motion and $J$-stability and then state some equivalent characterizations of $J$-stability for natural families of finite type entire functions.

\begin{defi}[{Holomorphic motion}]
Let $X$ be a subset of $\EC$, a map $h:\Lambda\times X\rightarrow\EC$ is called a \textit{holomorphic motion} of $X$ parameterized by the connected complex manifold $\Lambda$ and with base point $\lambda_0$ if
\begin{enumerate}
\item for every $z\in X$, $\lambda\mapsto h_\lambda(z)=h(\lambda,z)$ is holomorphic for $\lambda\in\Lambda$;
\item for every $\lambda\in\Lambda$, $z\mapsto h_\lambda(z)=h(\lambda,z)$ is injective on $X$; and
\item $h_{\lambda_0}(z)=h(\lambda_0,z)=z$ for all $z\in X$.
\end{enumerate}
\end{defi}

A basic fact about holomorphic motions is the \textit{$\lambda$-lemma}: any holomorphic motion $h:\Lambda\times X\rightarrow\EC$ can be extended to a holomorphic motion of its closure $h:\Lambda\times \overline{X}\rightarrow\EC$ (see \cite{MSS83}, \cite{Lyu83b}).

\begin{defi}[${J}$-stability]
Let $(f_\lambda)_{\lambda\in\Lambda}$ be a holomorphic family of transcendental entire functions, where $\Lambda$ is a connected complex manifold.
Given $\lambda_0\in \Lambda$, the map $f_{\lambda_0}$ is $J$-\textit{stable} if there exists a neighbourhood $U\subset \Lambda$ of $\lambda_0$ over which the Julia sets $J_\lambda:=J(f_\lambda)$ move holomorphically, i.e., there is a holomorphic motion $h: U\times J_{\lambda_0}\to\C$ such that $h_\lambda(J_{\lambda_0})=J_\lambda$ and $h_\lambda\circ f_{\lambda_0}= f_{\lambda}\circ h_\lambda$ holds on the Julia set $J_{\lambda_0}$.
\end{defi}

\begin{defi}[Natural family]
A holomorphic family $\{f_\lambda\}_{\lambda\in\Lambda}$ of finite type entire functions is a \emph{natural family} if it has the form $f_\lambda=\varphi_\lambda \circ f_{\lambda_0}\circ \psi^{-1}_\lambda$, where $f_{\lambda_0}$ is a finite type entire function, and $\varphi_\lambda,\psi_\lambda:\EC\to\EC$ are quasiconformal homeomorphisms depending holomorphically on $\lambda \in\Lambda$ with $\varphi_\lambda(\infty)=\psi_\lambda(\infty)=\infty$.
\end{defi}

Let $\{f_\lambda\}_{\lambda\in\Lambda}$ be a holomorphic family of finite type entire functions and $v_\lambda$ (resp., $c_\lambda$) be a singular value (resp., a critical point) of $f_\lambda$ depending holomorphically on $\lambda$ near some $\lambda_0 \in\Lambda$. We say that $v_\lambda$ (resp., $c_\lambda$) is \textit{passive} at $\lambda_0$ if there exists a neighborhood $U$ of $\lambda_0$ in $\Lambda$ such that the family $\{\lambda\mapsto f_\lambda^{\circ n}(v_\lambda)\}_{n \in \N}$ is normal on $U$. We say that $v_\lambda$ (resp., $c_\lambda$) is \textit{active} if it is not passive.

For a natural family $\{f_\lambda\}_{\lambda\in\Lambda}$, the critical points and singular values move holomorphically since $\psi_\lambda$ maps the critical points of $f_{\lambda_0}$ to those of $f_\lambda$, and $\varphi_\lambda$ maps the singular values of $f_{\lambda_0}$ to those of $f_\lambda$.

The following equivalent characterizations of $J$-stability of natural family of finite type entire functions (actually proved for meromorphic functions) was proved by Astorg, Benini and Fagella (see \cite[Theorem E]{ABF23}).

\begin{thm}[{Characterizations of $J$-stability}]\label{thm:J-stability}
Let $\left\{f_\lambda\right\}_{\lambda \in\Lambda}$ be a natural family of finite type entire functions. Let $U \subset \Lambda$ be a simply connected domain. The following are equivalent:
\begin{enumerate}
\item the Julia set moves holomorphically over $U$ (i.e. $f_\lambda$ is $J$-stable for all $\lambda \in U$);
\item every singular value is passive on $U$;
\item the maximal period of attracting cycles is bounded on $U$;
\item the number of attracting cycles is constant in $U$;
\item for all $\lambda \in U$, $f_\lambda$ has no non-persistent parabolic cycles.
\end{enumerate}
\end{thm}

The \textit{bifurcation locus} of the holomorphic family $(f_\lambda)_{\lambda\in\Lambda}$ of entire functions is defined as
\begin{equation}
\Bif(f_\lambda) := \{\lambda_0\in \Lambda \mid f_{\lambda_0} \text{ is not } J\text{-stable}\}.
\end{equation}
In view of Theorem \ref{thm:J-stability}, if $(f_\lambda)_{\lambda\in\Lambda}$ is a natural family of finite type entire functions,  it makes sense to define the bifurcation locus $\operatorname{Bif}(f_\lambda)$ as the set of parameters $\lambda$ for which any of the five conditions in Theorem \ref{thm:J-stability} is not satisfied.

According to \cite{EL92}, $J$-stable parameters of the natural family $(f_\lambda)_{\lambda\in\Lambda}$ of finite type entire functions form an open and dense set in $\Lambda$.  This is well known for rational maps by \cite{MSS83} and is also true for finite type meromorphic functions (see \cite[Corollary F]{ABF23}).
Hence $\Bif(f_\lambda)$ is a closed and nowhere dense subset of $\Lambda$.

\subsection{Polynomial-like maps}

Let $U$ and $V$ be two Jordan disks in $\C$ such that $U$ is compactly contained in $V$. The map $g:U\to V$ is called a \textit{polynomial-like map} of degree $d\geq 2$ if $g$ is a proper holomorphic surjection with degree $d$. The sets
\begin{equation}
K(g):=\bigcap\nolimits_{n\geq 0}g^{-n}(V) \text{\quad and\quad} J(g):=\partial K(g)
\end{equation}
are the \emph{filled Julia set} and the \textit{Julia set} of the polynomial-like map respectively.
Two polynomial-like maps $g_1:U_1\to V_1$ and $g_2:U_2\to V_2$ are said to be \emph{hybrid equivalent} if there is a quasiconformal mapping $h$ defined from a neighborhood of $K(g_1)$ onto that of $K(g_2)$, which conjugates $g_1$ to $g_2$ and the complex dilatation of $h$ on $K(g_1)$ is zero.
The following result is due to Douady and Hubbard (see \cite[p.\,296]{DH85b}).

\begin{thm}[Straightening Theorem]\label{thm:straightening}
Let $g:U\to V$ be a polynomial-like map of degree $d\geq 2$. Then $g:U\to V$ is hybrid equivalent to a polynomial $P$ with the same degree $d$. Moreover, if $K(g)$ is connected, then $P$ is uniquely determined up to a conjugation by an affine map.
\end{thm}

Denote $q_c(z)=z^d+c$, where $d\geq2$. The generalized \textit{Mandelbrot set} (or Multibrot set) is defined as
\begin{equation}
M_d:=\{c\in\C: \{q_c^{\circ n}(0)\}_{n\in\N} \text{ is bounded}\}.
\end{equation}
The traditional \textit{Mandelbrot set} is the quadratic version $M_2$.
Based on introducing holomorphic family of polynomial-like maps, Douady and Hubbard gave a proof of the appearance of quasiconformal copies of Mandelbrot sets in the parameter spaces of various one-dimensional analytic families \cite{DH85b}.

\subsection{Universality of the Mandelbrot set}

By applying the theory of polynomial-like maps further, McMullen proved that the bifurcation of any holomorphic families of rational maps is either empty or contains the quasiconformal image of $\partial M_d$ for some $d\geq 2$ \cite{McM00b}.
The proof of McMullen is purely local and can be applied to some holomorphic families of transcendental meromorphic functions.
Indeed, based on the idea of the proof of \cite[Theorem 4.1]{McM00b}, Astorg, Benini and Fagella genealized McMullen's result to natural families of finite type meromorphic functions. In this paper, we only state their result for transcendental entire functions. See Lemma 7.4, Theorem 7.5 and Corollary 1.2 of \cite{ABF26}.

\begin{thm}[Universality of the Mandelbrot set]\label{thm-univ}
Let $(f_\lambda)_{\lambda\in\Lambda}$ be a natural family of finite type entire functions with a marked active critical point $c_\lambda$, where $\Lambda \subset \mathbb{C}$ is a domain. Suppose $c_{\lambda}$ is not persistently a Picard exceptional value.
Then for any $\lambda_0\in\Bif(f_\lambda)$ and any neighborhood $W$ of $\lambda_0$,
\begin{enumerate}
\item there exists a quasiconformal embedding $\chi:\partial M_d\to \Bif(f_\lambda)\cap W$ for some $d\geq 2$; and
\item there exists $n\geq 1$ such that for any $c \in M_d$, there exists a polynomial-like map $f_\lambda^{\circ n}: U \rightarrow V$ which is hybrid conjugate to $q_c(z)=z^d+c$, where $\lambda\in W$.
\end{enumerate}
\end{thm}

Both of the proofs in \cite{McM00b} and \cite{ABF26} are based on \textit{Misiurewicz bifurcation}.
A parameter $\lambda\in\Lambda$ is a \textit{Misiurewicz point} for the pair $(f_\lambda, c_\lambda)$ if the forward orbit of the critical point $c_\lambda$ under $f_\lambda$ lands on a repelling periodic cycle. If the forward orbit of $c_\lambda$ does not coincide with any other critical points of $f_\lambda$, then the number $d$ in Theorem \ref{thm-univ} is the local degree of $c_\lambda$. A Misiurewicz bifurcation of degree $d\geq 2$ gives rise to a cascade of families of polynomial-like maps \(f_{\lambda}^{\circ n}: U_{0} \to U_{n}\), indexed by the return time \(n\) near a Misiurewicz point $\lambda_0$.

Misiurewicz bifurcation requires that the local degree of $f_\lambda$ at $c_\lambda$ is invariant and this is automatically satisfied by the definition of natural family, since the marked critical point $c_\lambda$ has constant degree. Another requirement of the Misiurewicz bifurcation is that the critical point $c_{\lambda_0}$ is \textit{unramified} for $f_{\lambda_0}$, which cannot be satisfied if $c_{\lambda_0}$ is a Picard exceptional value, where $\lambda_0$ is a Misiurewicz point. Hence in Theorem \ref{thm-univ}, $c_{\lambda}$ is required not to be persistently a Picard exceptional value, which may happen, for example, $f_\lambda(z)=\lambda e^{z^2}$, where $\lambda\in\C\setminus\{0\}$.

\section{Carpet Julia sets of hyperbolic functions}\label{sec:carpet-hole}

In this section, by using Bergweiler-Fagella-Rempe's criterion on the local connectivity of Julia sets of hyperbolic entire functions, we give the proofs of Theorem \ref{thm:carpet-hyper} and Corollary \ref{cor:carpet-hole}.

\subsection{Proof of Theorem \ref{thm:carpet-hyper}}

The Julia set of any transcendental entire function is an unbounded closed subset of $\C$. It is known that a continuum cannot fail to be locally connected only at a single point (see \cite[Corollary 5.13, p.\,78]{Nad92}). Therefore, the Julia set $J(f)$ of a transcendental entire function $f$ is locally connected in $\C$ if and only if $J(f)\cup\{\infty\}$ is locally connected in $\EC$.
In view of the local connectivity of the Julia set $J(f)$ of a transcendental entire function $f$ has been studied extensively recently, we use the following criterion to prove that $J(f)$ is a Sierpi\'{n}ski carpet.

\begin{lem}\label{lem:carpet-criterion}
Let $f$ be a transcendental entire function with non-empty Fatou set $F(f)$. Suppose
\begin{enumerate}
\item $J(f)$ is locally connected; and
\item the components of $F(f)$ are bounded by pairwise disjoint Jordan curves.
\end{enumerate}
Then $J(f)$ is a Sierpi\'{n}ski carpet.
\end{lem}

\begin{proof}
In order to prove that $J(f)$ is a Sierpi\'{n}ski carpet, according to \cite{Why58}, it suffices to verify the following five properties:
\begin{enumerate}
\item[(1)] $J(f)\cup\{\infty\}$ is a compact subset of $\EC$;
\item[(2)] $J(f)$ has empty interior;
\item[(3)] $J(f)$ is connected;
\item[(4)] $J(f)$ is locally connected; and
\item[(5)] $F(f)$ consists of infinitely many components which are bounded by pairwise disjoint Jordan curves.
\end{enumerate}
If $f$ has only finitely many Fatou components, then iterating $f$ several times if necessary, these Fatou components are completely invariant and they must be unbounded. However, if $f$ has an unbounded Fatou component, then $J(f)$ cannot be locally connected (see \cite{BD00b}, \cite{Osb13}), which contradicts to condition (a). Hence by the assumptions, we only need to check (3).
In fact, by \cite[Theorem B]{BD00b}, if $J(f)$ is not connected, then $J(f)$ must be not locally connected, which also violates the assumption. Thus $J(f)$ is a Sierpi\'{n}ski carpet.
\end{proof}

\begin{rmk}
Note that Lemma \ref{lem:carpet-criterion} not only can be applied to hyperbolic transcendental entire functions, but also to some others, especially to those functions whose Julia sets are known to be locally connected.
\end{rmk}

\begin{proof}[Proof of Theorem \ref{thm:carpet-hyper}]
Let $f\in\MB$ be a hyperbolic transcendental entire function without asymptotic values. Then the Fatou set of $f$ is non-empty.
Assume further that every component of $F(f)$ contains at most one critical value and the multiplicity of the critical points of $f$ is uniformly bounded. By \cite[Corollaries 1.3 and 1.9]{BFR15}, all Fatou components of $f$ are bounded and $J(f)$ is connected and locally connected. According \cite[Theorem 1.10]{BFR15}, all Fatou components of $f$ are Jordan domains (actually are quasi-disks).

To show that $J(f)$ is a Sierpi\'{n}ski carpet, by Lemma \ref{lem:carpet-criterion}, it is sufficient to prove that all boundaries of the Fatou components of $f$ are pairwise disjoint.
By the characterization of hyperbolic entire functions (see \cite[Proposition 2.1]{BFR15}), $F(f)$ is a finite union of attracting basins and hence all Fatou component of $f$ are eventually periodic.
By the assumption, the components of immediate attracting basins of $f$ have pairwise disjoint boundaries. Iterating $f$ finitely times if necessary, we assume that all immediate attracting basins of $f$ have period one.

Let $U'$ and $U''$ be two different components of $F(f)$. If $U'$ and $U''$ are preimages of different immediate attracting basins, then $\partial U'\cap\partial U''=\emptyset$.
Hence we assume that $U'$ and $U''$ are preimages of the same immediate attracting basin $U$, where $m\geq n\geq 0$ are smallest integers such that $f^{\circ m}(U')=f^{\circ n}(U'')=U$. Assume that $\partial U'\cap\partial U''\neq\emptyset$.
If $m>n\geq 0$, then $f^{\circ (m-1)}(U')$ is a Jordan domain which is different from $U$ and $\partial (f^{\circ (m-1)}(U'))\cap\partial (f^{\circ (m-1)}(U''))=\partial (f^{\circ (m-1)}(U'))\cap\partial U\neq\emptyset$. This implies that $\partial U$ contains a critical point, which is impossible since $f$ is hyperbolic. If $m=n\geq 1$, then $f^{\circ (m-1)}(U')$ and $f^{\circ (m-1)}(U'')$ are two different Jordan domains satisfying $\partial (f^{\circ (m-1)}(U'))\cap\partial (f^{\circ (m-1)}(U''))\neq\emptyset$. This implies that $\partial (f^{\circ (m-1)}(U'))$ contains a critical point, which is also impossible. Therefore, $\partial U'\cap\partial U''=\emptyset$. By the arbitrariness of $U'$ and $U''$, we conclude that $J(f)$ is a Sierpi\'{n}ski carpet.
\end{proof}


\subsection{Sierpi\'{n}ski holes in parameter spaces}

Let $\left\{f_\lambda\right\}_{\lambda \in\Lambda}$ be a natural family of finite type entire functions.
For any connected component $\MH$ of $\Lambda\setminus\Bif(f_\lambda)$, according to \cite[Proposition 5, p.\,1016]{EL92} (see also Theorem \ref{thm:J-stability}), if $f_{\lambda_0}\in\MH$ is hyperbolic, then all $f_\lambda\in\MH$ are hyperbolic. The open set $\MH$ is called a \textit{hyperbolic component} of the family $\left\{f_\lambda\right\}_{\lambda \in\Lambda}$.

\begin{defi}[Sierpi\'{n}ski holes]
Let $\MH$ be a hyperbolic component of the natural family $\left\{f_\lambda\right\}_{\lambda \in\Lambda}$ of finite type entire functions. If $J(f_{\lambda})$ is a Sierpi\'{n}ski carpet for some $\lambda\in\MH$, then $\MH$ is called a \textit{Sierpi\'{n}ski hole}.
\end{defi}

Corollary \ref{cor:carpet-hole} is a special case of the following result (see Figures \ref{Fig:M-sin-2} and \ref{Fig:carpet-2}).

\begin{thm}\label{thm:carpet-sin-d}
For the family
\begin{equation}
f_\lambda(z)=\lambda\sin^d z,
\end{equation}
where $d\geq 2$ and $\lambda\in\C\setminus\{0\}$, if $\lambda$ lies in the attracting basin of $0$ but not in the immediate attracting basin, then $J(f_\lambda)$ is a Sierpi\'{n}ski carpet. In particular, in the parameter plane of $\{f_\lambda\}_{\lambda\in\C\setminus\{0\}}$,
\begin{enumerate}
\item there is a Sierpi\'{n}ski hole $\MH_k$ containing $\lambda_k=k\pi$ for each $k\in\Z\setminus\{0\}$; and
\item there is a Sierpi\'{n}ski hole $\MH_k'$ containing $\lambda_k':=k\pi+\frac{\pi}{2}$ with $k\in\Z$ if and only if $k\in\Z\setminus\{-1,0\}$.
\end{enumerate}
Moreover, $\{\MH_k: k\in\Z\setminus\{0\}\}$ and $\{\MH_k': k\in\Z\setminus\{-1,0\}\}$ are pairwise different.
\end{thm}

\begin{figure}[!htpb]
  \setlength{\unitlength}{1mm}
  \subfigure[{Julia set of $\pi\sin^2 z$}]{\includegraphics[width=0.47\textwidth]{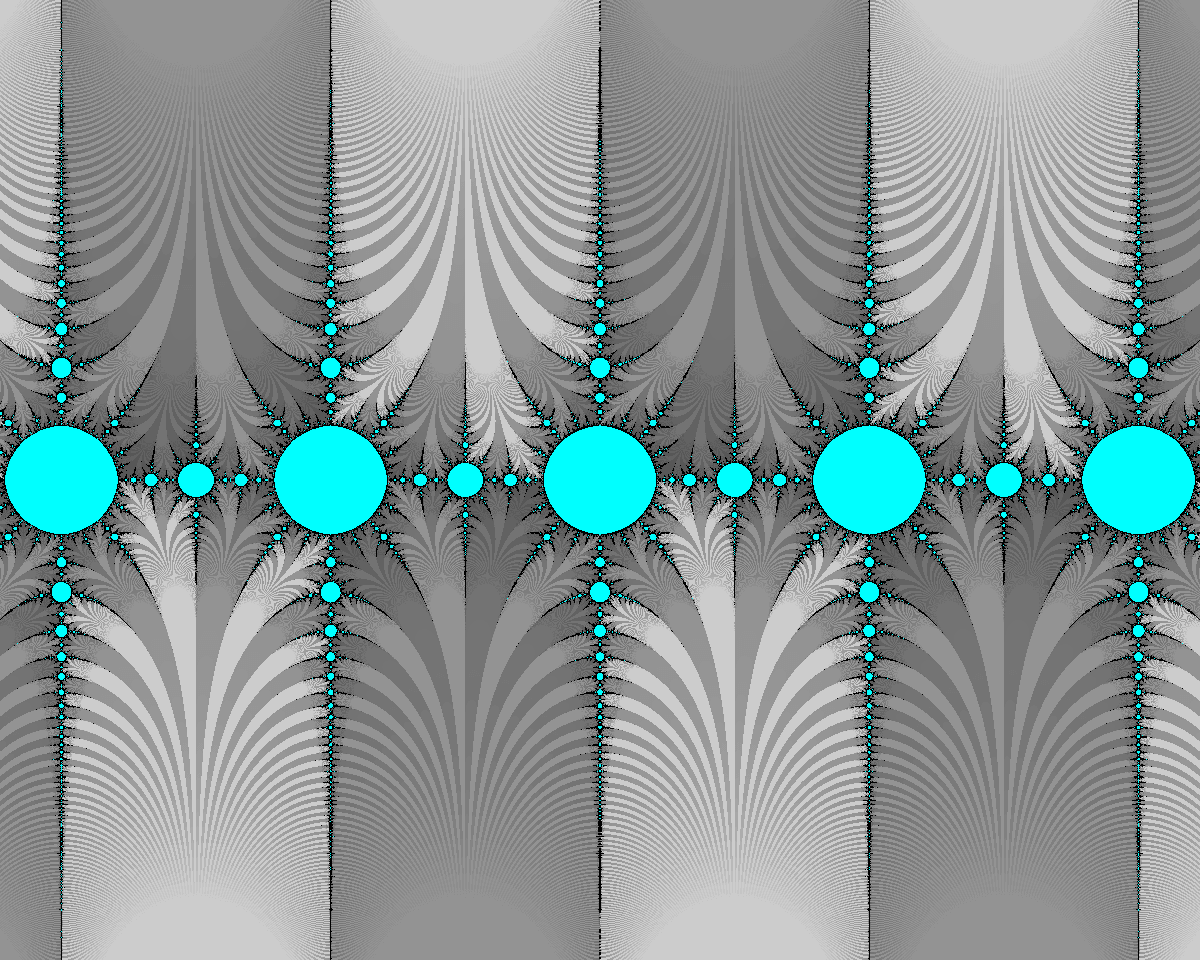}} \quad
  \subfigure[{Julia set of $\pi\sin^3 z$}]{\includegraphics[width=0.47\textwidth]{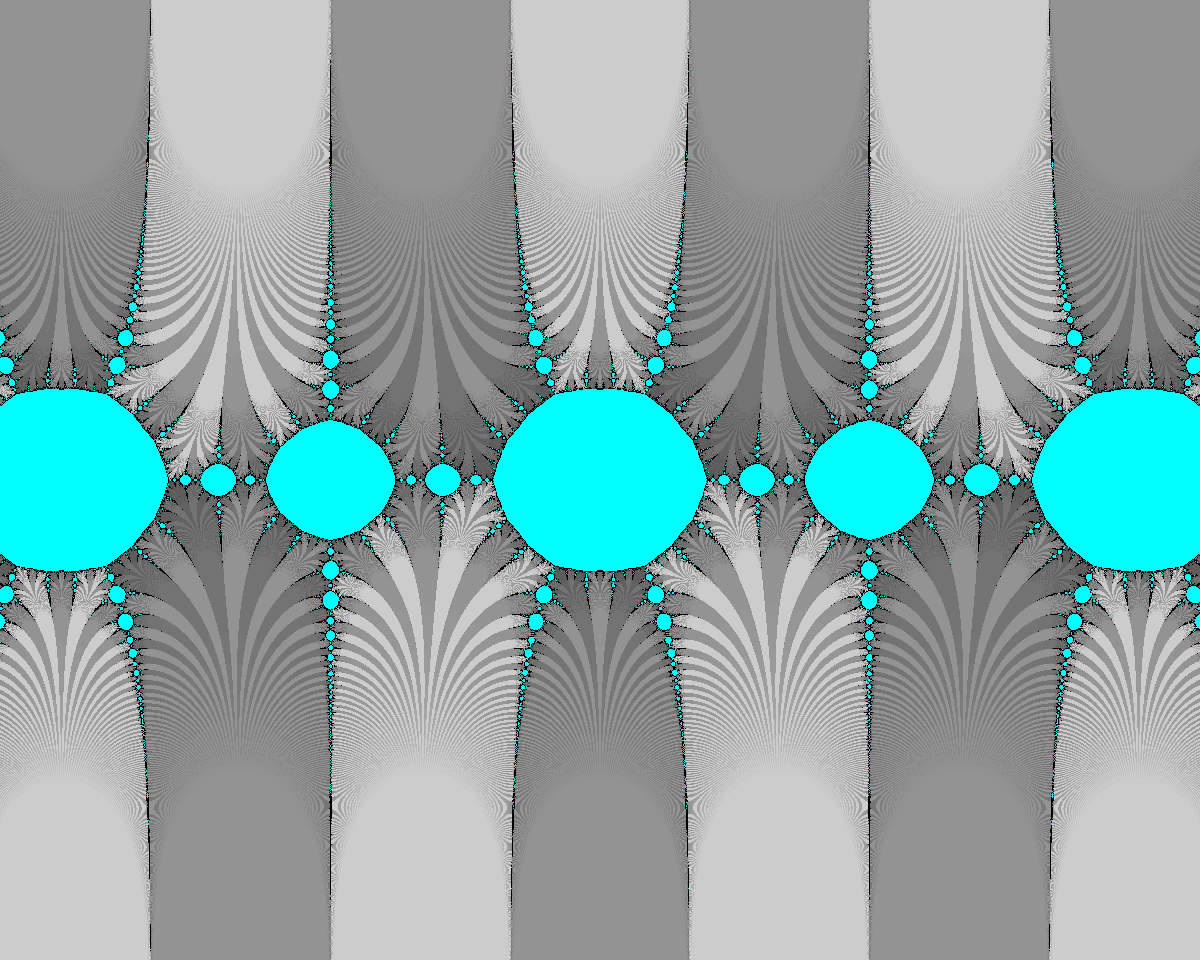}}
  \subfigure[{Julia set of $\frac{3\pi}{2}\sin^2 z$}]{\includegraphics[width=0.47\textwidth]{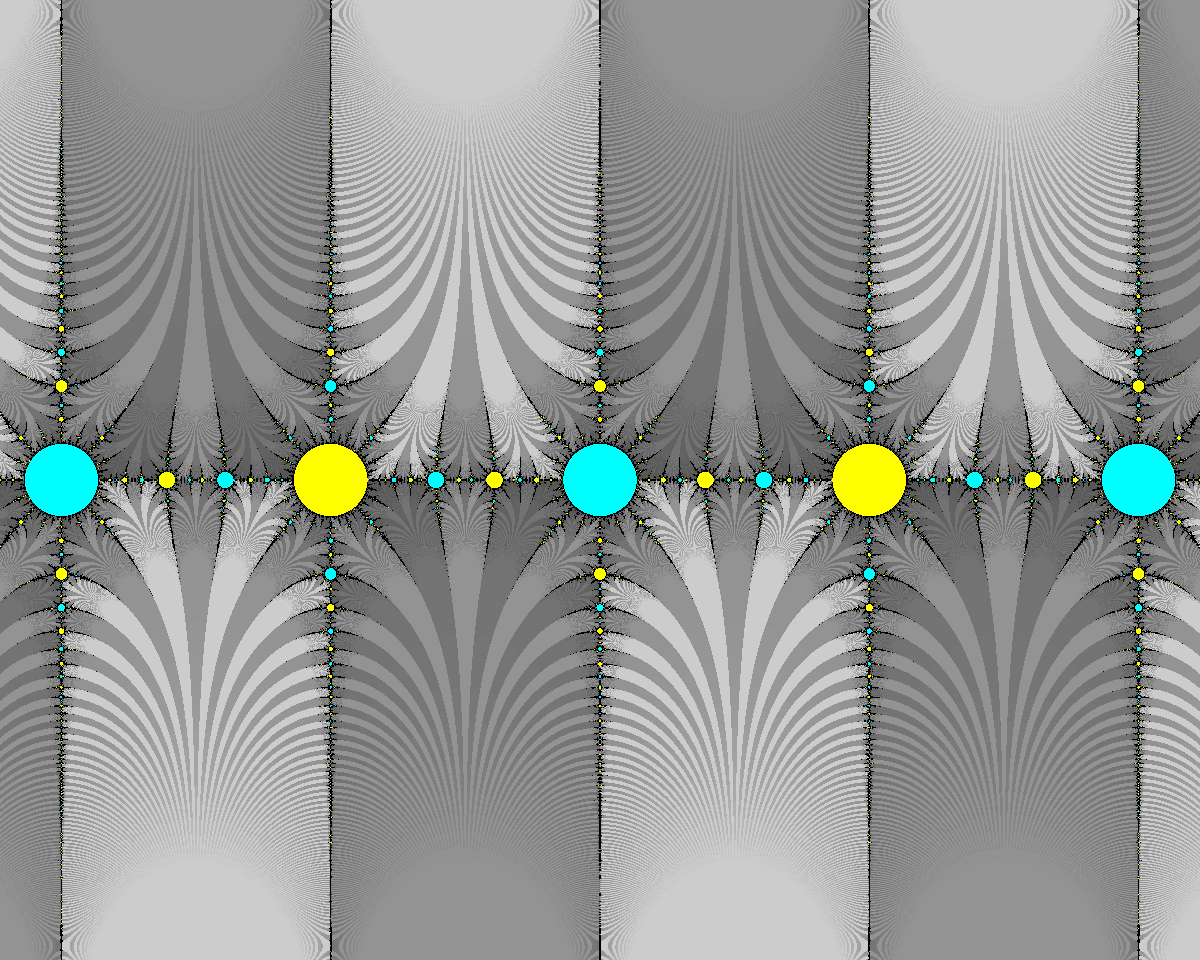}} \quad
  \subfigure[{Julia set of $\frac{\pi}{2}\sin^2 z$}]{\includegraphics[width=0.47\textwidth]{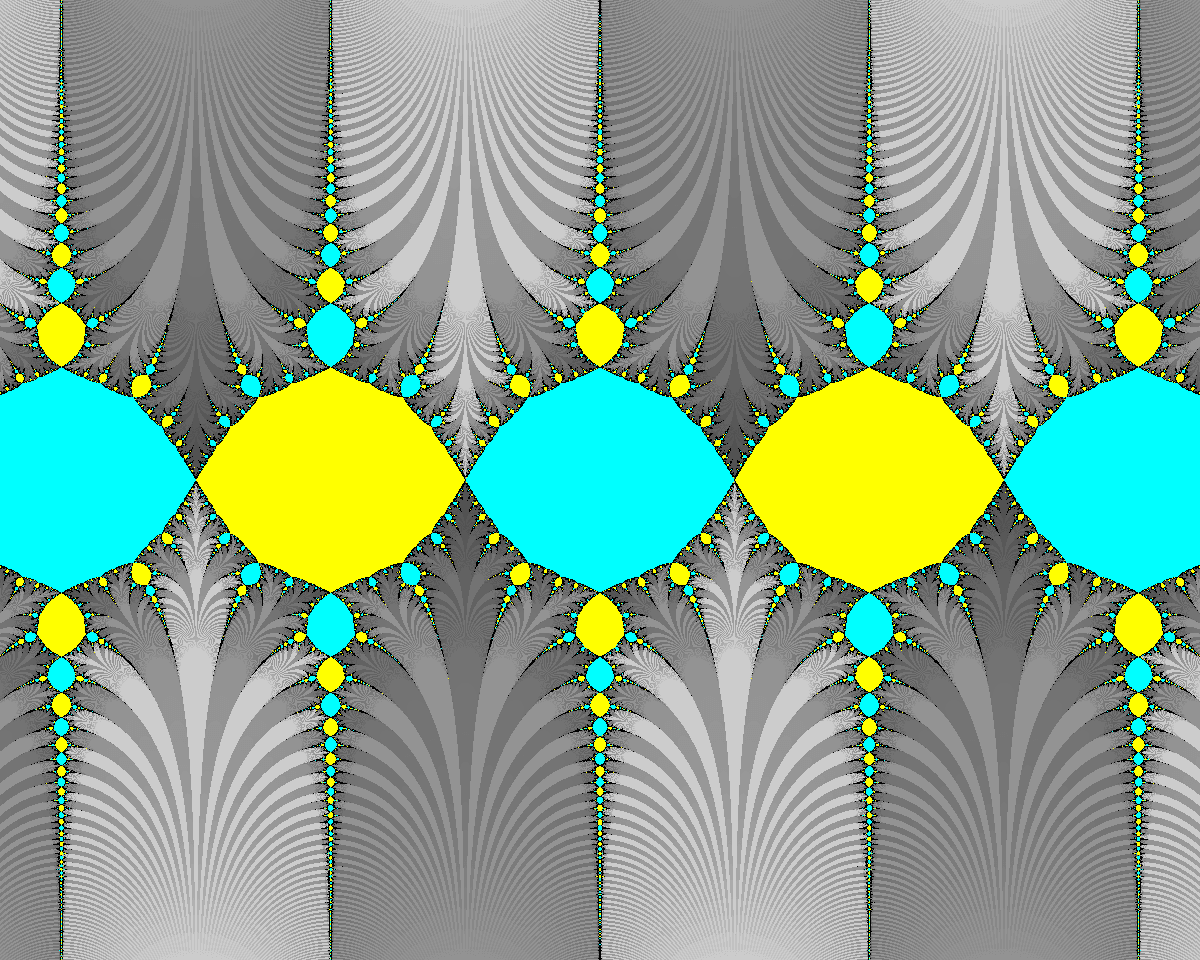}}
  \caption{Some Julia sets of $f_\lambda(z)=\lambda\sin^d z$, where $d\geq 2$. Theorem \ref{thm:carpet-sin-d} proves that the Julia set of $k\pi\sin^d z$ is a Sierpi\'{n}ski carpet for all $k\in\Z\setminus\{0\}$, and the Julia set of $(k\pi+\frac{\pi}{2})\sin^d z$ with $k\in\Z$ is a Sierpi\'{n}ski carpet if and only if $k\in\Z\setminus\{-1,0\}$. The cyan and yellow regions represent two different attracting basins. Range of figures: $[-3.5,3.5]\times[-2.8,2.8]$.}
  \label{Fig:carpet-2}
\end{figure}

\begin{proof}
Note that $f_\lambda$ is a natural family having a super-attracting fixed point at $0$. Each map in this family has no asymptotic values in $\C$ and the set of singular values (actually critical values) of $f_\lambda$ is $S(f_\lambda)=\{0,\lambda\}$ if $d\geq 2$ is even and $S(f_\lambda)=\{0,\lambda,-\lambda\}$ if $d\geq 3$ is odd. Moreover, the set of critical points of $f_\lambda$ is $\Crit(f_\lambda)=\{k\pi,k\pi+\frac{\pi}{2}:k\in\Z\}$ and the local degrees of $k\pi$ and $k\pi+\frac{\pi}{2}$ are $d$ and $2$ respectively.

\medskip
\textit{Step 1. Preperiod implies carpets}.
Suppose that the critical value $\lambda$ is in the attracting basin of $0$ but not in the immediate attracting basin $U_0$ of $0$. Then $-\lambda$ is also in the attracting basin of $0$ and hence $f_\lambda$ is hyperbolic.
If $d\geq 2$ is even, then $J(f_\lambda)$ is a Sierpi\'{n}ski carpet by Theorem \ref{thm:carpet-hyper}.
Let $d\geq 3$ be odd. Then the critical value $-\lambda$ is also not in $U_0$ by the symmetry since $f_\lambda$ is an odd function. We claim that $\lambda$ and $-\lambda$ are in different Fatou components of $f_\lambda$. Indeed, if there is a Fatou component $U_1$ containing $\{\lambda,-\lambda\}$, then by the dynamical symmetry, $U_1=-U_1$ is a Fatou component surrounding $U_0$. In particular, $U_1$ is not simply connected.
However, every connected component of an attracting basin is simply connected by the maximum principle (see also \cite[Proposition 2.1]{BFR15}).
This is a contradiction. Hence the critical values $\lambda$ and $-\lambda$ are in different Fatou components of $f_\lambda$. Thus $J(f_\lambda)$ is a Sierpi\'{n}ski carpet by Theorem \ref{thm:carpet-hyper}.

For any $\lambda$ which is in the attracting basin of $0$ but not in the immediate attracting basin, there exists a Sierpi\'{n}ski hole $\MH$ containing $\lambda$.
In the following we prove that there are infinitely many such Sierpi\'{n}ski holes in the parameter plane of $\{f_\lambda\}$.

\medskip
\textit{Step 2. Sierpi\'{n}ski holes containing $\lambda_k=k\pi$}.
For $k\in\Z\setminus\{0\}$, we have
\begin{equation}
f_{\lambda_k}(z)=\lambda_k\sin^d z=k\pi\sin^d z
\end{equation}
and $f_{\lambda_k}(\lambda_k)=f_{\lambda_k}(-\lambda_k)=0$. This implies that $\lambda_k$ lies in the attracting basin of $0$ and there exists a hyperbolic component $\MH_k$ containing $\lambda_k$.
A direct calculation shows that there exists a repelling fixed point $z_0\in(0,\frac{\pi}{2})$ if $k>0$ and a repelling periodic point $z_0\in(-\frac{\pi}{2},0)$ if $k<0$. Hence $z_0$ and $-z_0$ are contained in $J(f_{\lambda_k})$. Note that the dynamics of $f_{\lambda_k}$ is symmetric about the real axis $\R$ (since $\lambda_k$ is real):
\begin{itemize}
\item $J(f_{\lambda_k})$ is symmetric about $\R$; and
\item any Fatou component of $f_{\lambda_k}$ intersecting $\R$ is symmetric about $\R$.
\end{itemize}
By the maximum principle, we conclude that the Fatou components containing $0$, $\lambda_k$ and $-\lambda_k$ are different. Hence $\lambda_k$ is not in the immediate attracting basin of $0$ and $\MH_k$ is a Sierpi\'{n}ski hole for all $k\in\Z\setminus\{0\}$.

\medskip
\textit{Step 3. Sierpi\'{n}ski holes containing $\lambda_k'=k\pi+\frac{\pi}{2}$}. For $k\in\Z$, we have
\begin{equation}
f_{\lambda_k'}(z)=\lambda_k'\sin^d z=(k\pi+\tfrac{\pi}{2})\sin^d z.
\end{equation}
Hence
\begin{equation}
\left\{
\begin{array}{ll}
f_{\lambda_k'}(\lambda_k')=\lambda_k' \\
f_{\lambda_k'}(-\lambda_k')=(-1)^d\lambda_k'
\end{array}
\right.
\text{for even } k\text{\quad and }
\left\{
\begin{array}{ll}
f_{\lambda_k'}(\lambda_k')=(-1)^d\lambda_k' \\
f_{\lambda_k'}(-\lambda_k')=\lambda_k'
\end{array}
\right.
\text{ for odd } k.
\end{equation}
If $d$ is even, then $f_{\lambda_k'}(\lambda_k')=f_{\lambda_k'}(-\lambda_k')=\lambda_k'$ is a super-attracting fixed point. If $d$ is odd, then $\{\lambda_k',-\lambda_k'\}$ are super-attracting fixed points of $f_{\lambda_k'}$ (for even $k$) or they form a super-attracting cycle of period two (for odd $k$). In all cases, $f_{\lambda_k'}$ is hyperbolic and hence there exists a hyperbolic component $\MH_k'$ containing $\lambda_k'$. Still by the dynamical symmetry and maximum principle, the points $0$, $\lambda_k'$ and $-\lambda_k'$ are in different Fatou components of $f_{\lambda_k'}$. By \cite[Corollary 1.9]{BFR15}, $J(f_{\lambda_k'})$ is locally connected for all $k\in\Z$.

We first assume that $k\in\Z\setminus\{-1,0\}$. Note that $f_{\lambda_k'}(m\pi+\frac{\pi}{2})=(-1)^{dm}\lambda_k'$, where $m\in\Z$. Hence $m\pi+\frac{\pi}{2}\in F(f_{\lambda_k'})$ for all $m\in\Z$.
Let $U_m$ be the Fatou component of $f_{\lambda_k'}$ containing $m\pi+\frac{\pi}{2}$, where $m\in\Z$.
Note that $f_{\lambda_k'}(m\pi)=0$ for each $m\in\Z$ and
\begin{equation}
-|\lambda_k'|<-\pi<-\tfrac{\pi}{2}<0<\tfrac{\pi}{2}<\pi<|\lambda_k'|.
\end{equation}
Since $0$, $\lambda_k'$ and $-\lambda_k'$ are in different Fatou components of $f_{\lambda_k'}$, by the symmetry of $f_{\lambda_k'}$ (the Julia set is symmetric about the real axis) and the maximal principle, we conclude that each $U_m$ is a bounded Jordan domain and $\{\partial U_m:m\in\Z\}$ are pairwise disjoint. Hence $J(f_{\lambda_k'})$ is a Sierpi\'{n}ski carpet by Theorem \ref{thm:carpet-hyper} and $\MH_k'$ is a Sierpi\'{n}ski hole for all $k\in\Z\setminus\{-1,0\}$.

Now we prove that the Julia set of $f_{\lambda_k'}$ with $k\in\{-1,0\}$ (i.e., $f_{\pm\frac{\pi}{2}}$) is not a Sierpi\'{n}ski carpet. We only consider $f_{\frac{\pi}{2}}$ since the proof for $f_{-\frac{\pi}{2}}$ is completely similar. Let $V_0$ and $V_1$ be the Fatou components of $f_{\frac{\pi}{2}}$ containing $0$ and $\frac{\pi}{2}$ respectively.
Note that $x\mapsto \frac{\pi}{2}\sin^d x$ is increasing on $[0,\frac{\pi}{2}]$.
By considering the convexity of $x\mapsto \frac{\pi}{2}\sin^d x-x$ on $[0,\frac{\pi}{2}]$, a direct calculation shows that there exists a repelling fixed point $x_0\in(0,\frac{\pi}{2})$ such that $[0,x_0)\subset V_0$ and $(x_0,\frac{\pi}{2}]\subset V_1$. Hence $x_0\in\partial V_0\cap \partial V_1$ and $J(f_{\frac{\pi}{2}})$ is not a Sierpi\'{n}ski carpet.
In particular, $\MH_k'$ is not a Sierpi\'{n}ski hole for  $k\in\{-1,0\}$.

\medskip
\textit{Step 4. These Sierpi\'{n}ski holes are different}.
Note that each $\MH_k$ (resp., $\MH_k'$) is symmetric about the real axis. By the maximum principle, we conclude that each $\MH_k$ (resp., $\MH_k'$) is simply connected.
Then $\MH_k\neq \MH_l'$ for any $k\in\Z\setminus\{0\}$ and $l\in\Z\setminus\{-1,0\}$. This implies that $\MH_{k_1}\neq\MH_{k_2}$ for any different $k_1,k_2\in \Z\setminus\{0\}$ and $\MH_{l_1}'\neq\MH_{l_2}'$ for any different $l_1,l_2\in \Z\setminus\{-1,0\}$.
Thus $\{\MH_k: k\in\Z\setminus\{0\}\}$ and $\{\MH_l': l\in\Z\setminus\{-1,0\}\}$ are pairwise different.
The proof is complete.
\end{proof}

\subsection{The sine family}

The sine family $f_\lambda(z)=\lambda\sin z$, where $\lambda\in\C\setminus\{0\}$, was first systematically studied by Dom\'{i}nguez and Sienra in \cite{DS02}.
In this subsection, we prove that there also exist infinitely many Sierpi\'{n}ski holes in the parameter plane of this family.

\begin{thm}\label{thm:sine}
Let $f_\lambda(z)=\lambda\sin z$, where $\lambda\in\C\setminus\{0\}$ and $\lambda_k':=k\pi+\frac{\pi}{2}$ with $k\in\Z$. Then $J(f_{\lambda_k'})$ is a Sierpi\'{n}ski carpet if and only if $k\in\Z\setminus\{-1,0\}$.
\end{thm}

\begin{proof}
We only give a sketch of the proof since the idea is similar to Theorem \ref{thm:carpet-sin-d}.
Note that $f_\lambda$ is a natural family having no asymptotic values in $\C$ and the set of critical values of $f_\lambda$ is $S(f_\lambda)=\{\lambda,-\lambda\}$. Moreover, the set of critical points is $\Crit(f_\lambda)=\{k\pi+\frac{\pi}{2}:k\in\Z\}$ and the local degree of each critical point is two.

For each $k\in\Z$, we have
\begin{equation}
f_{\lambda_k'}(z)=\lambda_k'\sin z=(k\pi+\tfrac{\pi}{2})\sin z.
\end{equation}
Note that $f_{\lambda_k'}(\pm\lambda_k')=\pm\lambda_k'$ for even $k$ and $f_{\lambda_k'}(\pm\lambda_k')=\mp\lambda_k'$ for odd $k$. Hence $\{\lambda_k',-\lambda_k'\}$ are super-attracting fixed points of $f_{\lambda_k'}$ (for even $k$) or they form a super-attracting cycle of period two (for odd $k$). In both cases, $f_{\lambda_k'}$ is hyperbolic, and $\lambda_k'$, $-\lambda_k'$ are in different Fatou components of $f_{\lambda_k'}$. By \cite[Corollary 1.9]{BFR15}, $J(f_{\lambda_k'})$ is locally connected for all $k\in\Z$.

Let $k\in\Z\setminus\{-1,0\}$. Similar to Step 3 of Theorem \ref{thm:carpet-sin-d}, we have $m\pi+\frac{\pi}{2}\in F(f_{\lambda_k'})$ for all $m\in\Z$, each $U_m$ is a bounded Jordan domain and $\{\partial U_m:m\in\Z\}$ are pairwise disjoint, where $U_m$ is the Fatou component of $f_{\lambda_k'}$ containing $m\pi+\frac{\pi}{2}$.
Hence $J(f_{\lambda_k'})$ is a Sierpi\'{n}ski carpet by Theorem \ref{thm:carpet-hyper}.

Finally we prove that the Julia set of $f_{\pm\frac{\pi}{2}}$ is not a Sierpi\'{n}ski carpet. We only consider $f_{\frac{\pi}{2}}$ since the proof for $f_{-\frac{\pi}{2}}$ is completely similar. Let $U_\pm$ be the Fatou component containing $\pm\frac{\pi}{2}$ respectively. Since $0$ is a repelling fixed point of $f_{\frac{\pi}{2}}$, by the symmetry and maximal principle, we have $U_+\neq U_-$.
A direct calculation shows that $(0,\frac{\pi}{2}]\subset U_+$ and $[-\frac{\pi}{2},0)\subset U_-$. Hence $0\in\partial U_+\cap \partial U_-$. Therefore, $J(f_{\frac{\pi}{2}})$ is not a Sierpi\'{n}ski carpet.
The proof is complete.
\end{proof}

\begin{rmk}
One may verify that the Julia set of $z\mapsto k\pi\sin z$ is $\C$ for all $k\in\Z\setminus\{0\}$, which is different from $z\mapsto k\pi\sin^d z$ with $d\geq 2$.
\end{rmk}

\section{Density of Sierpi\'{n}ski carpet parameters}\label{sec:density}

In this section, by using Alhamed-Rempe-Sixsmith's criterion on the local connectivity of Julia sets of strongly geometrically finite transcendental entire functions, we provide a proof of Theorem \ref{thm:carpet-dense} and also give some applications.

\subsection{Proof of Theorem \ref{thm:carpet-dense}}

By Theorem \ref{thm:straightening}, every polynomial-like map $g:U\to V$ is hybrid equivalent to a unicritical polynomial $q_c(z)=z^d+c$, where $c\in\C$.
If the unique critical orbit of $g$ is contained in $U$, then $q_c$ is unique.

Let $f$ be a transcendental entire function.
If there exist $p\geq 1$ and two Jordan domains $U$, $V$ such that $g=f^{\circ p}|_U:U\to V$ is a polynomial-like map with connected Julia set, then $f$ is called \textit{$p$-renormalizable} (or \textit{renormalizable} in short).
The sets $K(g)$, $f(K(g))$, $\cdots$, $f^{\circ (p-1)}(K(g))$ are called \textit{small filled Julia sets}, and $J(g)$, $f(J(g))$, $\cdots$, $f^{\circ (p-1)}(J(g))$ are called \textit{small Julia sets}.
For more backgrounds and results on the polynomial-like renormlization, we refer to \cite[Chapter 5]{McM94b}.

\begin{proof}[Proof of Theorem \ref{thm:carpet-dense}]
Let $(f_\lambda)_{\lambda\in\Lambda}$ be a natural family of finite type entire functions without asymptotic values, where $c_{\lambda}$ is a marked active critical point which is not persistently a Picard exceptional value and $f_\lambda(c_\lambda)$ is the only active critical value.
Let $\lambda_0$ be any parameter in $\Bif(f_\lambda)$ and $W$ be any open neighborhood of $\lambda_0$.
By Theorem \ref{thm-univ}, there exist $d\geq 2$ and $n\geq 1$ such that for any $c \in M_d$, there exists a polynomial-like map $f_\lambda^{\circ n}: U \rightarrow V$ which is hybrid conjugate to $q_c(z)=z^d+c$, where $\lambda=\lambda(c)\in W$.
Let $c\in \partial M_d$ be a parabolic parameter such that $q_c(z)=z^d+c$ has infinitely many bounded Fatou components and all of them are bounded by pairwise disjoint Jordan curves.
We use $\MU:=\{U_1,\cdots, U_{n'}\}$ to denote the corresponding immediate parabolic basins of $f_\lambda$ with $\lambda=\lambda(c)\in W\cap\Bif(f_\lambda)$.
Then $\{\partial U_1,\cdots, \partial U_{n'}\}$ are pairwise disjoint.
In the following we prove that $J(f_\lambda)$ is a Sierpi\'{n}ski carpet.

According to \cite{ARS22}, a transcendental entire function $f$ is \textit{strongly geometrically finite} if $F(f)\cap S(f)$ is compact, $J(f)\cap \MP(f)$ is finite and $f$ has bounded criticality on the Julia set (i.e., $J(f)$ contains no finite asymptotic values of $f$, and the local degree of $f$ at points of $J(f)$ is uniformly bounded).
By the assumptions of the theorem, $f_\lambda$ is strongly geometrically finite with no finite asymptotic values, and the Fatou set of $f_\lambda$ contains at most two critical values and they are in different Fatou components. By \cite[Corollary 1.10]{ARS22}, $J(f_\lambda)$ is locally connected.
If all critical values, except $f_\lambda(c_\lambda)$, are all preperiodic in the Julia set of $f_\lambda$, then $F(f_\lambda)$ consists of $\MU$ and their preimages.
Similar to the proof of Theorem \ref{thm:carpet-hyper}, the boundaries of all components of $F(f_\lambda)$ are pairwise disjoint and $J(f_\lambda)$ is a Sierpi\'{n}ski carpet by Lemma \ref{lem:carpet-criterion}.

Suppose the number of the remaining critical values of $f_\lambda$ is one and this critical value is contained in a fixed immediate attracting or parabolic basin $U_0$ of $f_\lambda$. We claim that $\partial U_i\cap\partial U_0=\emptyset$ for all $1\leq i \leq n'$. Note that the small filled Julia set $K(g)$ of $g=f_\lambda^{\circ n}|_U: U \rightarrow V$ in $U$ is locally connected. By the maximal principle, we conclude that $K(g)\cap \partial U_0$ contains at most one point.
Assume that $\partial U_i\cap \partial U_0\neq\emptyset$ for some $i$. Let $z_i$ be the only element of $\partial U_i\cap \partial U_0$.
Without loss of generality, we assume that $U_i\subset K(g)$.
Note that $f_\lambda^{\circ n}(U_i)=U_{j}\subset K(g)$ and $\partial U_i\cap \partial U_j=\emptyset$ for some $1\leq j \leq n'$.
Then $\partial (f_\lambda^{\circ n}(U_i))\cap \partial (f_\lambda^{\circ n}(U_0))=\partial U_j\cap \partial U_0\subset K(g)\cap \partial U_0$ is a point which is different from $z_i$. This is a contradiction.
Hence $\partial U_i\cap\partial U_0=\emptyset$ for all $1\leq i \leq n'$. Similarly as above, $J(f_\lambda)$ is a Sierpi\'{n}ski carpet. By the arbitrariness of $\lambda_0$ and $W$, we conclude that the Sierpi\'{n}ski carpet parameters are dense in $\Bif(f_\lambda)$.
\end{proof}

\begin{rmk}
In the proof of Theorem \ref{thm:carpet-dense}, we choose parabolic parameters in the bifurcation locus. In fact, we can also choose Siegel parameters (see \S\ref{sec:carpet-Siegel}). Moreover, if the natural family has a fixed immediate attracting or parabolic basin for all $\lambda$, then we can choose Misiurewicz parameters in the bifurcation locus.
\end{rmk}

\subsection{Various cosine families}

In this subsection, we consider
\begin{equation}
g(z)=a e^z+be^{-z}+c,
\end{equation}
where $a,b\in\C\setminus\{0\}$ and $c\in\C$.
Obviously, $g$ has no finite asymptotic values and exactly two critical values $\pm 2\sqrt{ab}+c$.
As an immediate consequence of Theorem \ref{thm:carpet-dense}, we have the following result.

\begin{cor}[see Figure \ref{Fig:carpet-3}]\label{cor:carpet-density}
Suppose $(g_a)$ is one of the following family:
\begin{enumerate}
\item $g_a=g_{1,a}$ has a super-attracting fixed point at $0$:
\begin{equation}
g_{1,a}(z)=a (e^z+e^{-z}-2), \text{\quad where } a\in\C\setminus\{0\};
\end{equation}
\item $g_a=g_{2,a}$ has a parabolic fixed point at $0$ with multiplier $1$:
\begin{equation}
g_{2,a}(z)=a e^z+(a-1)e^{-z}+1-2a, \text{\quad where } a\in\C\setminus\{0,1\};
\end{equation}
\item $g_a=g_{3,a}$ has a preperiodic critical orbit $\log a\mapsto 0\mapsto 2\pi\ii\mapsto 2\pi\ii$:
\begin{equation}
g_{3,a}(z)=\frac{2\pi\ii}{(a-1)^2}(e^z+a^2e^{-z}-2a), \text{\quad where } a\in\C\setminus\{0,1\}.
\end{equation}
\end{enumerate}
Then the Sierpi\'{n}ski carpet parameters are dense in the bifurcation locus of $(g_a)$.
\end{cor}

\begin{figure}[!htpb]
  \setlength{\unitlength}{1mm}
  \centering
  \includegraphics[width=0.47\textwidth]{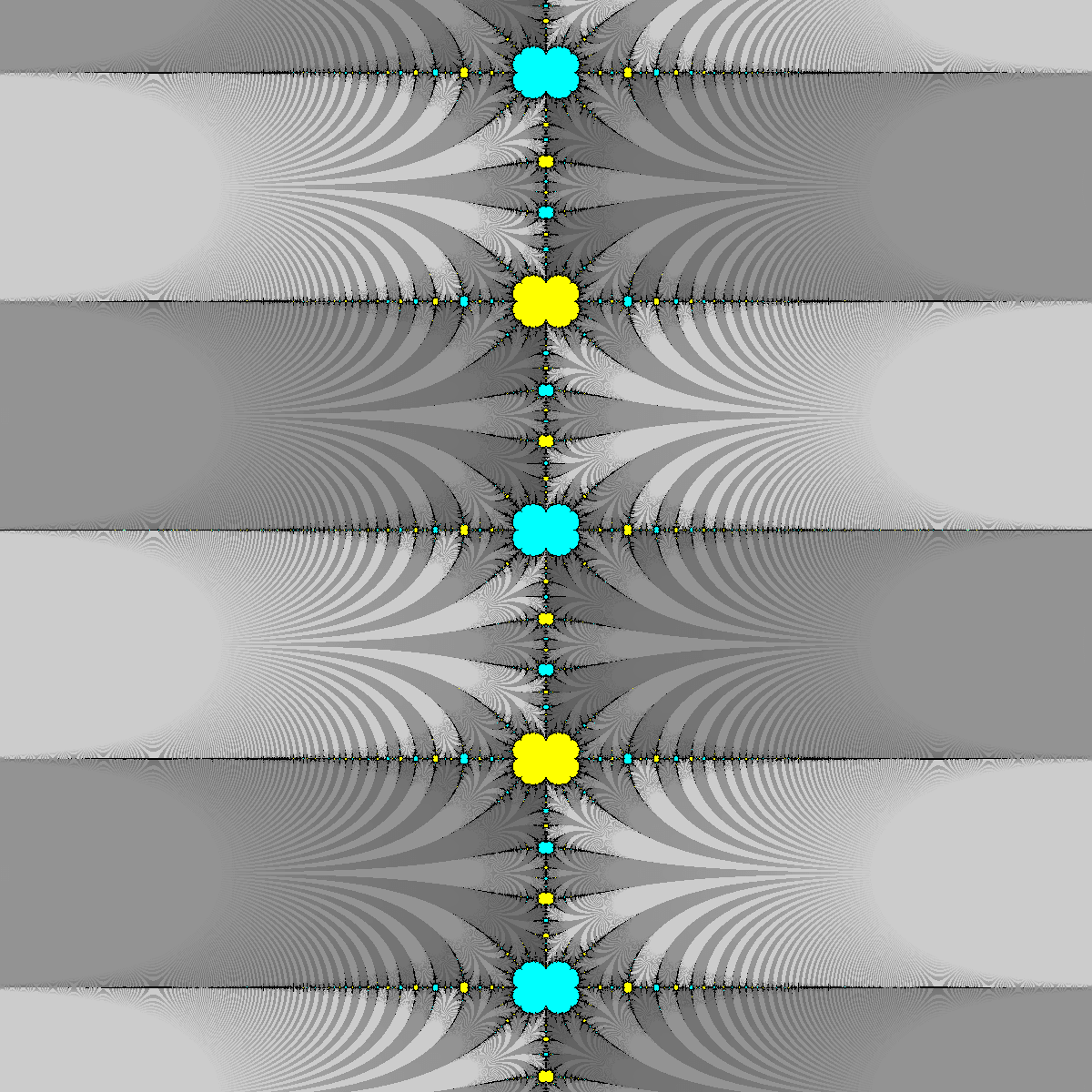} \quad
  \includegraphics[width=0.47\textwidth]{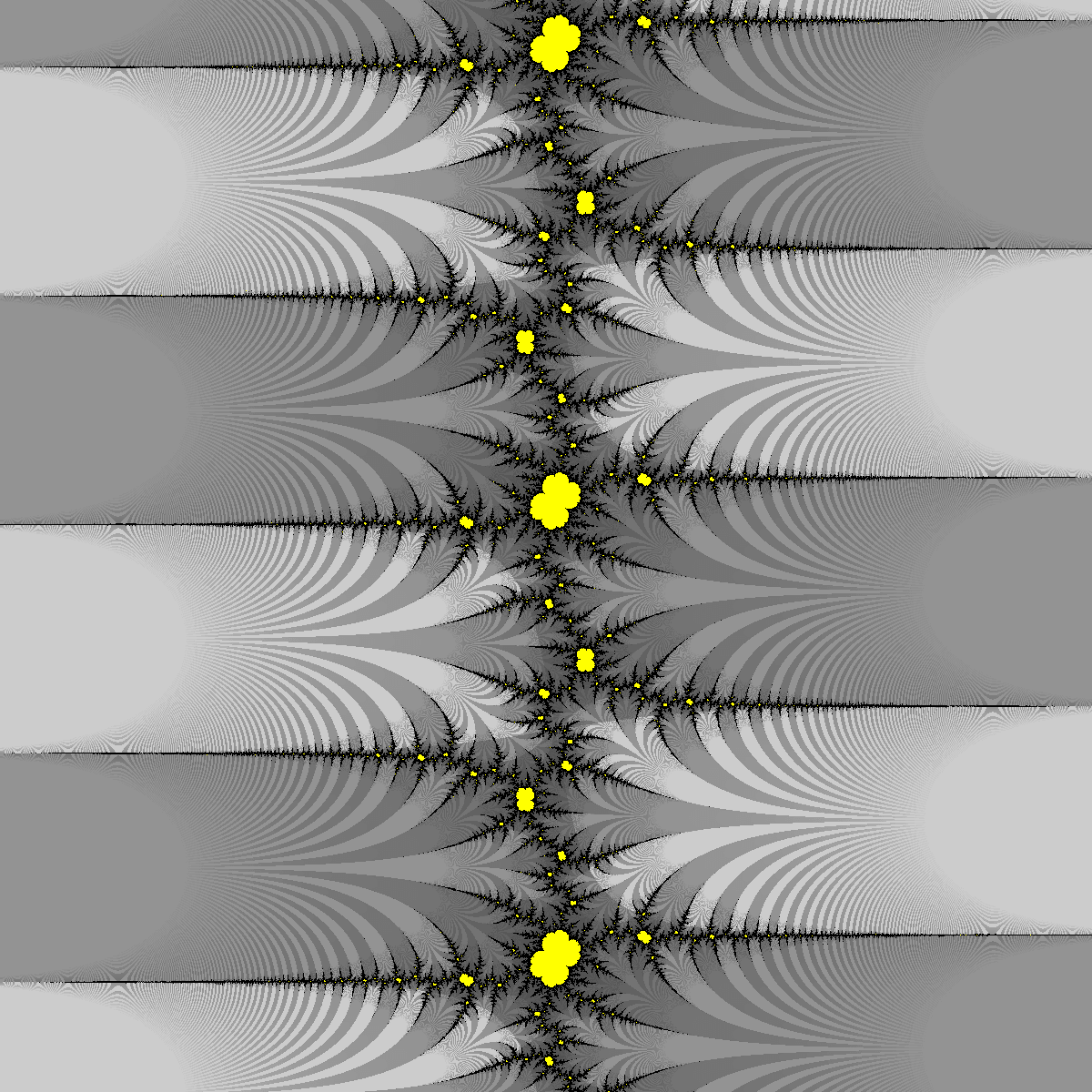}
  \caption{Julia sets of $g_{2,a_2}$ and $g_{3,a_3}$, where $a_2\approx0.5 + 2.24670472 \ii$ and $a_3\approx-2.52227895 - 1.79075517\ii$ are chosen such that $g_{2,a_2}$ has two parabolic cycles and $g_{3,a_3}$ has a parabolic cycle and a finite critical orbit in the Julia set. Both Julia sets are Sierpi\'{n}ski carpets. The cyan and yellow regions denote different parabolic basins. }
  \label{Fig:carpet-3}
\end{figure}

A direct calculation shows that the family $f_\lambda(z)=\lambda\sin^2(z)$ is conformally conjugate to the family $g_a(z)=a (e^z+e^{-z}-2)$. Hence Theorem \ref{thm:carpet-dense} can be applied to $f_\lambda$ (see Figures \ref{Fig:M-sin-2} and \ref{Fig:carpet-2}). Moreover, for any $d\geq 1$, the family $f_\lambda(z)=\lambda\sin^d z$ is semi-conjugate to
\begin{equation}
h_\lambda(z)=\lambda^2\sin^{2d}(\sqrt{z}), \text{\quad where }\lambda\in\C\setminus\{0\}.
\end{equation}
It is easy to verify that $h_\lambda$ has exactly two critical values $\{0,\lambda^2\}$ and Theorem \ref{thm:carpet-dense} can be also applied to $h_\lambda$. Hence the Sierpi\'{n}ski carpet parameters are dense in the bifurcation locus of $f_\lambda(z)=\lambda\sin^d z$ for all $d\geq 1$.

\section{Carpet Julia sets with Siegel disks}\label{sec:carpet-Siegel}

The local connectivity of transcendental entire functions with Siegel disks was studied in \cite{YZZ25}. These Siegel disks are of bounded type and their boundaries contain at least one critical point. Hence these Julia sets are not Sierpi\'{n}ski carpets. In this section, we shall ``put" some hairy Siegel disks in the transcendental Julia sets such that they are Sierpi\'{n}ski carpets.

\subsection{Quadratic hariy Siegel disks}

We first state some useful results on quadratic polynomials with Siegel disks.
For $\alpha\in\R\setminus\Q$, let
\begin{equation}
P_\alpha(z)=e^{2\pi\ii\alpha}z+z^2.
\end{equation}
If $P_\alpha$ can be locally linearizable at the origin, then the maximal linearization domain containing the origin is a simply connected Fatou component of $P_\alpha$, which is called the \textit{Siegel disk} of $P_\alpha$ centered at the origin. We denote it by $\Delta_\alpha$.

Any irrational number $\alpha\in\R\setminus\Q$ has a continued fractional expansion $\alpha=[a_0;a_1,a_2,\cdots]$, where $a_0\in\Z$ and $a_n\geq 1$ for all $n\geq 1$. Let $N\geq 2$ be a given integer. We denote
\begin{equation}\label{equ:high-type}
\HT_{N}:=\{ \alpha=[a_0;a_1,a_2,\cdots]: \ \forall\, n\ge1,\, a_n\ge N \}.
\end{equation}
The irrational numbers in $\HT_N$ are called of \textit{high type} (of $N$). The following result is proved in \cite{Che25}, and see also \cite{SY25} for a partial result.

\begin{thm}\label{thm:hairy-cycle}
There exists $N\geq 2$ such that for any Brjuno but non-Herman type number $\alpha \in \HT_{N}$, $\partial\Delta_\alpha$ is a Jordan curve which does not contain the critical point $-e^{2\pi\ii\alpha}/2$, and the post-critical set $\MP(P_\alpha)$ of $P_\alpha$ is a one-sided hairy circle which satisfies $\MP(P_\alpha)\cap (\bigcup_{k\geq 1}P_\alpha^{-k}(\overline{\Delta}_\alpha)\setminus\overline{\Delta}_\alpha)=\emptyset$.
\end{thm}

For definitions of Brjuno numbers, Herman numbers, and the precise definition of one-sided hairy circles, see \cite{Che25}. Roughly speaking, a \emph{one-sided hairy circle} is a collection
of arcs landing on a dense subset of a Jordan curve such that every arc in the collection is
accumulated from both sides by arcs in the collection.

\subsection{Carpet Julia sets with Siegel disks}

To prove the local connectivity of a compact set in $\EC$, the following criterion is useful (see {\cite[Theorem 4.4, p.\,113]{Why42}}).

\begin{lem}[LC criterion]\label{lem:LC-criterion}
A connected compact subset $X$ in $\EC$ is locally connected if and only if the following two conditions hold:
\begin{enumerate}
\item  the boundary of every component of $\EC\setminus X$ is locally connected; and
\item  for any given $\varepsilon>0$, there are only finitely many components of $\EC\setminus X$ whose spherical diameters are greater than $\varepsilon$.
\end{enumerate}
\end{lem}

Bergweiler and Morosawa proved the following result \cite[Theorem 1]{BM02}, which can be used to control the size of Fatou components of certain semi-hyperbolic entire functions. See \cite[Lemma 2.4]{YZZ25} for the present formulation.

\begin{lem}\label{lem:semi-hyperbolic}
Let $f$ be a transcendental entire function, and let $U_0, V_0$ be two bounded Jordan domains and $D_0\geq 1$. Suppose $V_0$ is not contained in any Siegel disk of $f$ and $U_0\Subset V_0$.
Then for any $\varepsilon>0$, there exists $N>0$ such that for any component $V_n$ of $f^{-n}(V_0)$ with $\deg(f^{\circ n}:V_n\to V_0)\leq D_0$, $\diam_{\EC}(U_n)<\varepsilon$ for all $n\geq N$, where $U_n$ is any connected component of $f^{-n}(U_0)$ contained in $V_n$.
\end{lem}

\begin{proof}[Proof of Theorem \ref{thm:carpet-Siegel}]
We will construct Sierpi\'{n}ski carpet Julia sets with Siegel disks in the sine family
\begin{equation}
f_\lambda(z)=\lambda\sin z, \text{\quad where }\lambda\in\C\setminus\{0\}.
\end{equation}
The critical values of $f_\lambda$ are $\lambda$ and $-\lambda$.
Let $\alpha$ be an irrational number in Theorem \ref{thm:hairy-cycle}.
Note that all critical points of $f_\lambda$ have local degree two.
By Theorem \ref{thm-univ}, there exist $p\geq 1$ and a polynomial-like map $f_{\lambda_0}^{\circ p}: U' \rightarrow V'$ which is hybrid conjugate to $P_\alpha(z)=e^{2\pi\ii\alpha}z+z^2$ such that $U'$ contains a Siegel disk $\Delta^{\lambda_0}$  and also the critical value $\lambda_0$. Since $f_\lambda$ is an odd function, by the symmetry, $f_{\lambda_0}^{\circ p}: -U'\rightarrow -V'$ is also a polynomial-like map which is hybrid conjugate to $P_\alpha$, where $-U'$ contains the Siegel disk $-\Delta^{\lambda_0}$  and also the critical value $-\lambda_0$. It is easy to see that $\partial\Delta^{\lambda_0}\cap\partial(-\Delta^{\lambda_0})=\emptyset$.

It is well know that $f_\lambda$ has neither Baker domains nor wandering domains for all $\lambda$ (see \cite{GK86} or \cite{EL92}). Hence all Fatou components of $f_\lambda$ are eventually periodic. Note that $f_{\lambda_0}$ has no attracting and parabolic basins. Thus each Fatou component of  $f_{\lambda_0}$ is eventually iterated to $\Delta^{\lambda_0}$ or $-\Delta^{\lambda_0}$. Since the Fatou set of $f_{\lambda_0}$ does not contain any critical values and $\Delta^{\lambda_0}$ is a Jordan domain, we conclude that all Fatou components of $f_{\lambda_0}$ are bounded Jordan domains by \cite[Proposition 2.9]{BFR15}.

In the following we prove that the spherical diameters of the Fatou components of $f_{\lambda_0}$ tend to zero.
To make the idea clear and for simplicity, we assume that $p=1$, i.e., $f_{\lambda_0}: U' \rightarrow V'$ is a polynomial-like map of degree two (the argument for $p\geq 2$ is completely similar).
By Theorem \ref{thm:hairy-cycle}, there exists a Jordan domain $V_0$ containing $\overline{\Delta^{\lambda_0}}$ such that
\begin{equation}
\overline{V}_0\cap\overline{V}_1=\emptyset \text{\quad and\quad}
\overline{V}_1\cap\MP(f_{\lambda_0})=\emptyset,
\end{equation}
where $V_1$ is any component of $f_{\lambda_0}^{-1}(V_0)\setminus V_0$ containing a component of $f_{\lambda_0}^{-1}(\Delta^{\lambda_0})\setminus \Delta^{\lambda_0}$.
Applying Lemma \ref{lem:semi-hyperbolic} to $U_0=\Delta^{\lambda_0}$, $V_0$ and $D_0=1$, we conclude that for any $\varepsilon>0$, there exists $N>0$ such that for any component $V_n$ of $f_{\lambda_0}^{-n}(V_0)\setminus\bigcup_{k=0}^{n-1}f_{\lambda_0}^{-k}(V_0)$ satisfying $f_{\lambda_0}^{\circ (n-1)}(V_n)\cap V_0=\emptyset$, we have $\deg(f_{\lambda_0}^{\circ n}: V_n\to V_0)= 1$ and $\diam_{\EC}(U_n)<\varepsilon$ for all $n\geq N$, where $U_n$ is any connected component of $f_{\lambda_0}^{-n}(U_0)\setminus\bigcup_{k=0}^{n-1}f_{\lambda_0}^{-k}(U_0)$ contained in $V_n$. By the symmetry of $f_{\lambda_0}$, we have the same conclusion for the Siegel disk $-\Delta^{\lambda_0}$ and its preimages.

Since $U_0=\Delta^{\lambda_0}$ is a Jordan domain, for any given $n\geq 1$, the Fatou components $U_n$'s of $f_{\lambda_0}$ satisfying $f_{\lambda_0}^{\circ n}(U_n)=U_0$ and $f_{\lambda_0}^{\circ (n-1)}(U_n)\neq U_0$ can only accumulate at infinity. Hence for given $n\geq 1$, there are only finitely many of $U_n$'s having spherical diameters which are greater than $\varepsilon$.
Thus there are only finitely many Fatou components whose spherical diameter is greater than $\varepsilon$.
By Lemma \ref{lem:LC-criterion}, $J(f_{\lambda_0})$ is locally connected. Since $\partial\Delta^{\lambda_0}\cap\partial(-\Delta^{\lambda_0})=\emptyset$, by a similar proof to Theorem \ref{thm:carpet-hyper}, it follows that all Fatou components of $f_{\lambda_0}$ are bounded by pairwise disjoint Jordan curves. Hence $J(f_{\lambda_0})$ is a Sierpi\'{n}ski carpet.
\end{proof}

\section{Carpet Julia sets with wandering domains}\label{sec:carpet-wandering}

In this section, we construct some transcendental entire functions having carpet Julia sets with wandering domains. Since the topology of the boundaries of wandering domains is not easy to study directly, inspired by \cite[\S 5]{Bak84}, we study the boundaries of periodic Fatou components of holomorphic self-maps of $\C^*:=\C\setminus\{0\}$ first and then use the logarithmic lift method to obtain the related results on wandering domains.

\begin{lem}[see Figure \ref{Fig:carpet-wandering}]\label{lem:puncture}
If $\lambda=\pm\sqrt{1+4k^2\pi^2}$ for $k\in\Z\setminus\{0\}$, then
\begin{equation}
g_\lambda(z)=z e^{\frac{1}{2}\lambda(z-\frac{1}{z})}:\C^*\to\C^*
\end{equation}
has two super-attracting fixed points on the unit circle and the Julia set of $g_\lambda$ is a Sierpi\'{n}ski carpet.
\end{lem}

\begin{proof}
Note that $g_\lambda$ has no asymptotic values in $\C^*$. For each $\lambda\neq 0$, a direct calculation shows
\begin{equation}
g_\lambda'(z)=\left(1+\frac{\lambda}{2}\Big(z+\frac{1}{z}\Big)\right)e^{\frac{1}{2}\lambda(z-\frac{1}{z})}.
\end{equation}
Hence $g_\lambda$ has exactly two critical points $(-1\pm\sqrt{1-\lambda^2})/\lambda$ of local degree $2$ if $\lambda\neq \pm 1$ and a critical point of local degree 3 if $\lambda=\pm 1$.

Considering the equations $g_\lambda'(z)=0$ and $g_\lambda(z)=z$, we obtain
\begin{equation}
\left\{
\begin{array}{ll}
1+\frac{\lambda}{2}\Big(z+\frac{1}{z}\Big)=0 \\
\frac{\lambda}{2}\Big(z-\frac{1}{z}\Big)=2k\pi\ii \text{ for some } k\in\Z.
\end{array}
\right.
\end{equation}
Hence we have the solutions
\begin{equation}
\left\{
\begin{array}{ll}
\lambda_+=\sqrt{1+4k^2\pi^2} \\
z_+=\frac{-1+2k\pi\ii}{\sqrt{1+4k^2\pi^2}}
\end{array}
\right.
\quad \text{or}\quad
\left\{
\begin{array}{ll}
\lambda_+'=-\sqrt{1+4k^2\pi^2} \\
z_+'=\frac{1-2k\pi\ii}{\sqrt{1+4k^2\pi^2}}.
\end{array}
\right.
\end{equation}
Therefore, if $\lambda=\lambda_+=\sqrt{1+4k^2\pi^2}$ (resp., $\lambda=\lambda_+'$) for $k\in\Z\setminus\{0\}$, then $g_\lambda$ has two super-attracting fixed points $\{z_+,\overline{z}_+\}$ (resp., $\{z_+',\overline{z}_+'\}$) on the unit circle.

We first prove that $J(g_\lambda)$ is locally connected for $\lambda=\lambda_+$ (the case for $\lambda=\lambda_+'$ is completely similar). We only give a sketch here since the proof is similar to that of hyperbolic entire functions in \cite{BFR15}.
Note that $g_\lambda:\C^*\to \C^*$ has exactly two singular values in $\C^*$ and hence is of finite type. Thus $g_\lambda$ has neither wandering domains (see \cite{Kot87}, \cite{Mak87} or \cite{Kee88}) nor Baker domains (see \cite[Theorem 1.1]{FM17}). Since $g_\lambda$ has exactly two super-attracting fixed points, obviously, $g_\lambda$ has no parabolic basins, Siegel disks and Herman rings. Hence $g_\lambda: \C^*\to \C^*$ is hyperbolic (i.e., the set of singular values is contained in a bounded annular neighborhood of the unit circle and in attracting basins).

Similar to \cite[Proposition 2.2]{BFR15} (see also \cite[Lemma 5.1]{Rem09b} and \cite[Proposition 3.4]{Mih10}), there exist a compact set $K$ of $\C^*$ containing two super-attracting fixed points $\{z_+,\overline{z}_+\}$ and a constant $\eta>1$, such that $g_\lambda(K) \subset \operatorname{int}(K)$ and
\begin{equation}
\|D g_\lambda(z)\|_W \geq \eta>1
\end{equation}
for all $z \in V$, where $W:=\C^* \setminus K$, $V:=g_\lambda^{-1}(W)$ and $\|D g_\lambda\|_W$ denotes the derivative of $g_\lambda$ with respect to the hyperbolic metric of $W$.
Let $U_+$ and $U_-$ be the immediate super-attracting basins containing $z_+$ and $z_-:=\overline{z}_+$ respectively. By a similar proof of \cite[Theorem 1.10]{BFR15}, there exists a sequence of parameterized continuous curves in $U_+$ (resp., $U_-$) converging uniformly to $\partial U_+$ (resp., $\partial U_-$) and hence $U_+$, $U_-$ are simply connected and bounded in $\C^*$, and moreover, their boundaries are locally connected. Similar to \cite[Corollary 1.8]{BFR15}, $J(g_\lambda)$ is locally connected.

Note that $g_\lambda(\T)=\T$, where $\T$ is the unit circle. Specifically, if $z=e^{\ii\theta}\in\T$, then
\begin{equation}
g_\lambda(e^{\ii\theta})=e^{\ii(\theta+\lambda\sin\theta)}.
\end{equation}
It is easy to see that $1$ (corresponds to $\theta=0$) and $-1$  (corresponds to $\theta=\pi$) are repelling fixed points of $g_\lambda$.
Since $\lambda=\lambda_+=\sqrt{1+4k^2\pi^2}>2k\pi\geq 2\pi$, there exist $\theta_1\in(0,\frac{\pi}{2})$ and $\theta_2\in(\frac{\pi}{2},\frac{3\pi}{2})$ such that $g_\lambda(e^{\ii\theta_1})=-1$ and $g_\lambda(e^{\ii\theta_2})=1$. Hence $e^{\ii\theta_1}$, $e^{\ii\theta_2}$, $1$ and $-1$ are not contained in $U_+\cup U_-$. By the symmetry (the Julia set of $g_\lambda$ is symmetric about the unit circle) and the maximum principle, $U_+$ and $U_-$ are Jordan domains and their boundaries are disjoint. This implies that all Fatou components of $g_\lambda$ are Jordan domains and their boundaries are pairwise disjoint. Thus $J(g_\lambda)$ is a Sierpi\'{n}ski carpet.
\end{proof}

\begin{rmk}
If $\lambda=1$ or $-1$ (i.e., $k=0$ in Lemma \ref{lem:puncture}), then $g_\lambda$ is still hyperbolic and $J(g_\lambda)$ is also locally connected but it is not a Sierpi\'{n}ski carpet. See \cite[Figure 4]{FH09}.
\end{rmk}

\begin{figure}[!htpb]
  \setlength{\unitlength}{1mm}
  \centering
  \includegraphics[width=0.47\textwidth]{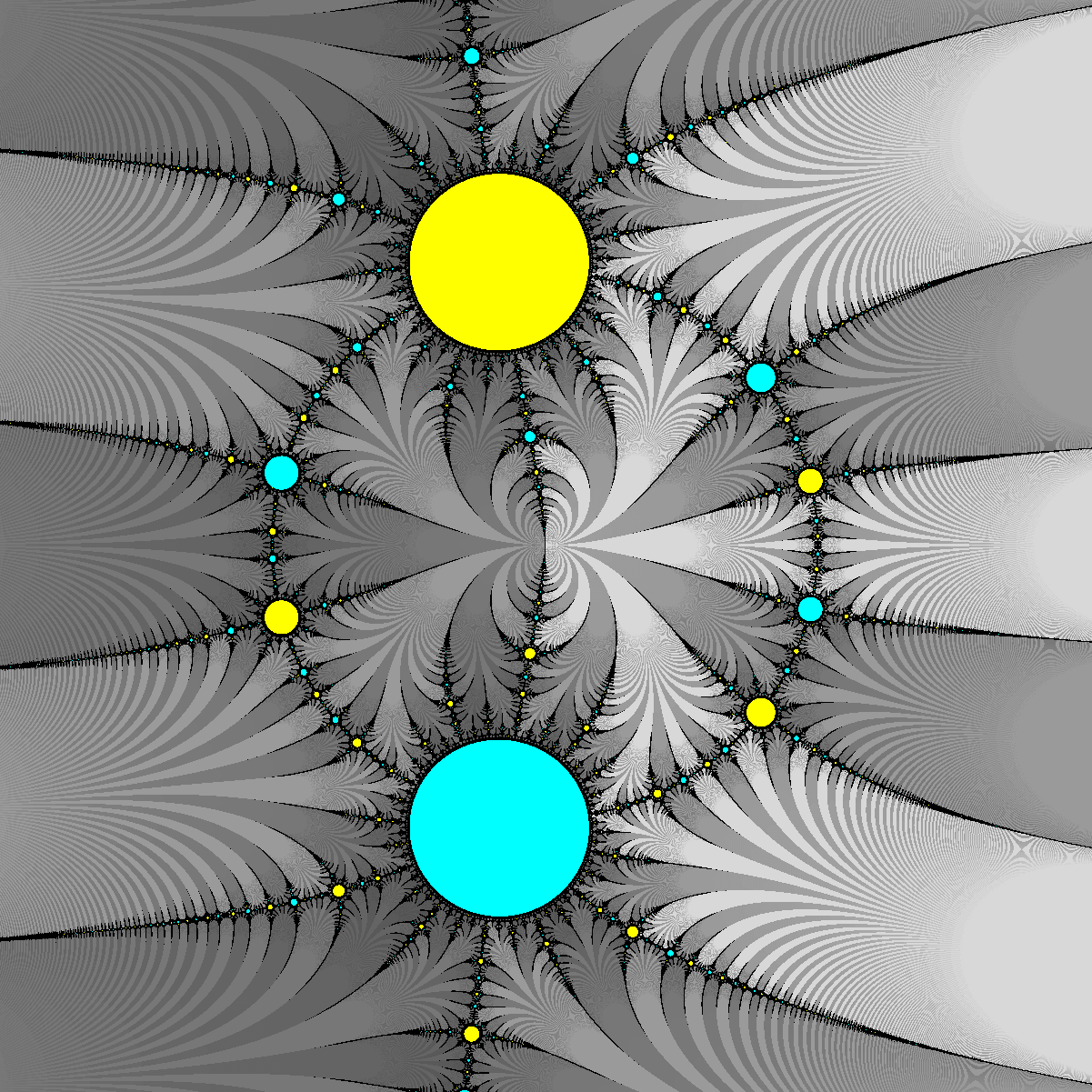} \quad
  \includegraphics[width=0.47\textwidth]{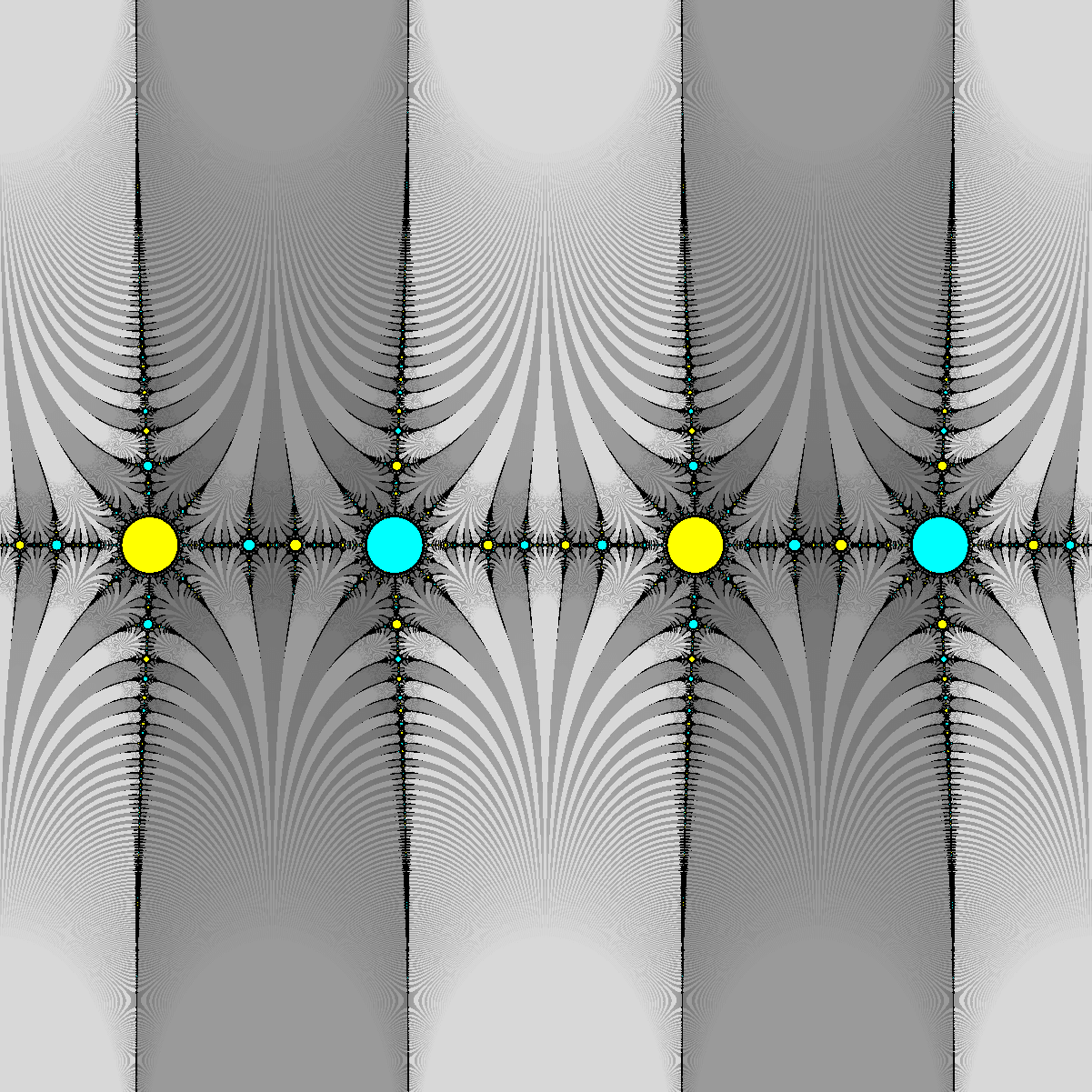}
  \caption{Left: The Julia set of $g_\lambda(z)=z e^{\frac{1}{2}\lambda(z-\frac{1}{z})}$ with $\lambda=\sqrt{1+4\pi^2}$, which is a Sierpi\'{n}ski carpet. Two super-attracting basins of $g_\lambda$ are marked yellow and cyan respectively. Right: The Julia set of $f_{0,\lambda}(z)=z+\sqrt{1+4\pi^2}\sin z$, which is also a Sierpi\'{n}ski carpet. The function $f_{0,\lambda}$ is a logarithmic lift of $g_\lambda$ and the Fatou set of $f_{0,\lambda}$ contains two grand orbits of wandering domains.}
  \label{Fig:carpet-wandering}
\end{figure}

Theorem \ref{thm:carpet-wandering} is an immediate consequence of the following result.

\begin{thm}\label{thm:carpet-lift}
For any $n\in\Z$ and $\lambda=\lambda_k:=\sqrt{1+4k^2\pi^2}$ with $k\in\Z_+$, the entire function
\begin{equation}
f_{n,\lambda}(z)=z+\lambda\sin z+2n\pi
\end{equation}
has wandering domains and its Julia set is a Sierpi\'{n}ski carpet.
\end{thm}

\begin{proof}
We use logarithmic lift of $g_\lambda$ to obtain $f_{n,\lambda}$ having wandering domains with Sierpi\'{n}ski carpet Julia set.
The argument is inspired by \cite[Example 5.3]{Bak84}.

\textit{Step 1. Logarithmic lift.} Recall that $g_\lambda(z)=z e^{\frac{1}{2}\lambda(z-\frac{1}{z})}$. Denote $\exp(z):=e^{\ii z}$. Then for any $n\in\Z$, we have the following commuting diagram
\[
\begin{CD}
\mathbb{C} @>f_{n,\lambda}>> \mathbb{C} \\
@VV\exp V        @VV \exp V \\
\mathbb{C}^* @>g_\lambda>> \mathbb{C}^*.
\end{CD}
\]
By \cite{Ber95}, we have $z\in J(f_{n,\lambda})$ if and only if $\exp(z)\in J(g_\lambda)$. Note that $f_{n,\lambda}^{\circ j}(z+2\pi)=f_{n,\lambda}^{\circ j}(z)+2\pi$ for all $j\in\N$. Hence
\begin{equation}
J(f_{m,\lambda})=J(f_{n,\lambda})=J(f_{n,\lambda})+2\pi \text{\quad for any } m, n\in\Z.
\end{equation}
By Lemma \ref{lem:puncture}, $J(f_{n,\lambda_k})$ is a Sierpi\'{n}ski carpet for any $n\in\Z$ and $k\in\Z_+$.
It suffices to prove that $f_{n,\lambda_k}$ has a wandering domain for any $n\in\Z$ and $k\in\Z_+$.

\medskip
\textit{Step 2. Lifting one super-attracting basin.}
Let $k$ be a positive integer. It is straightforward to verify that
\begin{equation}
f_{-k,\lambda_k}(z)=z+\sqrt{1+4k^2\pi^2}\sin z-2k\pi
\end{equation}
has a super-attracting fixed point at each
\begin{equation}
\widetilde{z}_m^+:=(2m+1)\pi-\arcsin\frac{2k\pi}{\sqrt{1+4k^2\pi^2}}, \text{ where }m\in\Z.
\end{equation}
Let $\widetilde{U}_m^+$ be the immediate super-attracting basin of $f_{-k,\lambda_k}$ containing $\widetilde{z}_m^+$. Then $U_+=\exp(\widetilde{U}_m^+)$ and $\widetilde{U}_m^+$ is a Jordan domain, where $U_+$ is the immediate super-attracting basin of $g_{\lambda_k}$ containing $z_+=\exp(\widetilde{z}_m^+)=\frac{-1+2k\pi\ii}{\sqrt{1+4k^2\pi^2}}$ (see Lemma \ref{lem:puncture}).

Since $J(f_{n,\lambda_k})$ is invariant under translation by $2\pi$, we conclude that $\{\widetilde{U}_m^+:m\in\Z\}$ are different Fatou components of $f_{-k+n,\lambda_k}$ for any $n\in\Z$.
Moreover, $f_{-k+n,\lambda_k}(\widetilde{U}_m^+)=\widetilde{U}_{m+n}^+$ and $\{\widetilde{U}_m^+:m\in\Z\}$ are wandering domains of $f_{-k+n,\lambda_k}$, where $n\neq 0$.

\medskip
\textit{Step 3. Lifting another super-attracting basin.}
For positive integer $k$, it is easy to verify that
\begin{equation}
f_{k,\lambda_k}(z)=z+\sqrt{1+4k^2\pi^2}\sin z+2k\pi
\end{equation}
has a super-attracting fixed point at each
\begin{equation}
\widetilde{z}_m^-:=(2m+1)\pi+\arcsin\frac{2k\pi}{\sqrt{1+4k^2\pi^2}}, \text{ where }m\in\Z.
\end{equation}
Let $\widetilde{U}_m^-$ be the immediate super-attracting basin of $f_{k,\lambda_k}$ containing $\widetilde{z}_m^-$. Then $U_-=\exp(\widetilde{U}_m^-)$ and $\widetilde{U}_m^-$ is a Jordan domain, where $U_-$ is the immediate super-attracting basin of $g_{\lambda_k}$ containing $z_-=\exp(\widetilde{z}_m^-)=\overline{z}_+=\frac{-1-2k\pi\ii}{\sqrt{1+4k^2\pi^2}}$.

Similarly, $\{\widetilde{U}_m^-:m\in\Z\}$ are different Fatou components of $f_{k+n,\lambda_k}$ for any $n\in\Z$.
In particular, $f_{k+n,\lambda_k}(\widetilde{U}_m^-)=\widetilde{U}_{m+n}^-$ and $\{\widetilde{U}_m^-:m\in\Z\}$ are wandering domains of $f_{k+n,\lambda_k}$, where $n\neq 0$.

\medskip
\textit{Step 4. The conclusion.}
Combining Steps 2 and 3 together, we conclude that for any positive integer $k$, there are following three cases:
\begin{itemize}
\item if $n=-k$, then $f_{n,\lambda_k}$ has immediate super-attracting basins $\{\widetilde{U}_m^+:m\in\Z\}$ and wandering domains $\{\widetilde{U}_m^-:m\in\Z\}$;
\item if $n=k$, then $f_{n,\lambda_k}$ has immediate super-attracting basins $\{\widetilde{U}_m^-:m\in\Z\}$ and wandering domains $\{\widetilde{U}_m^+:m\in\Z\}$; and
\item if $n\in\Z\setminus\{k,-k\}$, then $f_{n,\lambda_k}$ has wandering domains $\{\widetilde{U}_m^+:m\in\Z\}$ and $\{\widetilde{U}_m^-:m\in\Z\}$.
\end{itemize}
Therefore, $f_{n,\lambda_k}$ has wandering domains and $J(f_{n,\lambda_k})$ is a Sierpi\'{n}ski carpet for any $n\in\Z$ and $k\in\Z_+$. The proof is complete.
\end{proof}

\begin{rmk}
One may also lift the parabolic basin or Siegel disk of $g_\lambda$ to obtain transcendental entire functions having carpet Julia sets with wandering domains. We refer to \cite{FH09} and \cite{BEFRS22} for further study of this topic.
\end{rmk}

\bibliographystyle{amsalpha}
\bibliography{E:/Latex-model/Ref1}

\end{document}